\pdfoutput=1 

\documentclass[11pt]{amsart}

\usepackage{amssymb,amsthm,enumitem,colonequals,tikz-cd,microtype}
\usepackage[normalem]{ulem} 

\usepackage[osf]{Baskervaldx}
\usepackage[baskervaldx]{newtxmath}
\usepackage[cal=boondoxo]{mathalfa} 

\usepackage[top=3.75cm, bottom=3cm, left=3.5cm, right=3.5cm]{geometry}

\usepackage{xcolor}
\colorlet{darkblue}{blue!55!black}
\colorlet{darkcyan}{cyan!50!black}
\colorlet{darkgreen}{green!60!black}

\PassOptionsToPackage{hyphens}{url}
\usepackage{hyperref}
\hypersetup{
    colorlinks=true,
    linkcolor=darkblue,
    urlcolor=darkcyan,
    citecolor=darkgreen,
}

\def\eqref#1{\textcolor{darkblue}{(\ref{#1})}}

\usepackage[nameinlink]{cleveref} 
\Crefformat{section}{#2\S#1#3}
\Crefmultiformat{section}{#2\S\S#1#3}{ and~#2#1#3}{, #2#1#3}{, and~#2#1#3}

\usepackage[pagewise]{lineno}
\let\oldequation\equation
\let\oldendequation\endequation
\renewenvironment{equation}{\linenomathNonumbers\oldequation}{\oldendequation\endlinenomath}
\expandafter\let\expandafter\oldequationstar\csname equation*\endcsname
\expandafter\let\expandafter\oldendequationstar\csname endequation*\endcsname
\renewenvironment{equation*}{\linenomathNonumbers\oldequationstar}{\oldendequationstar\endlinenomath}
\let\oldalign\align
\let\oldendalign\endalign
\renewenvironment{align}{\linenomathNonumbers\oldalign}{\oldendalign\endlinenomath}
\expandafter\let\expandafter\oldalignstar\csname align*\endcsname
\expandafter\let\expandafter\oldendalignstar\csname endalign*\endcsname
\renewenvironment{align*}{\linenomathNonumbers\oldalignstar}{\oldendalignstar\endlinenomath}

\theoremstyle{plain}
\newtheorem{theorem}{Theorem}[section]
\newtheorem{lemma}[theorem]{Lemma}
\newtheorem{corollary}[theorem]{Corollary}
\newtheorem{proposition}[theorem]{Proposition}

\theoremstyle{definition}
\newtheorem{definition}[theorem]{Definition}

\newtheorem{remark}[theorem]{Remark}
\newtheorem{setup}[theorem]{Setup}
\newtheorem{notation}[theorem]{Notation}

\newtheorem*{ack}{Acknowledgments}

\AddToHook{env/notation/begin}{\crefalias{theorem}{notation}}
\AddToHook{env/lemma/begin}{\crefalias{theorem}{lemma}}
\AddToHook{env/corollary/begin}{\crefalias{theorem}{corollary}}
\AddToHook{env/proposition/begin}{\crefalias{theorem}{proposition}}
\AddToHook{env/definition/begin}{\crefalias{theorem}{definition}}
\AddToHook{env/remark/begin}{\crefalias{theorem}{remark}}
\AddToHook{env/setup/begin}{\crefalias{theorem}{setup}}
\AddToHook{env/example/begin}{\crefalias{theorem}{example}}

\numberwithin{equation}{section}
\numberwithin{theorem}{section}

\title[Gluing weak approximability along recollements]{Gluing weak approximability along recollements}

\author[P.~Lank]{Pat Lank}
\address{P.~Lank,
Dipartimento di Matematica “F. Enriques”, Universit\`{a} degli Studi di Milano, Via Cesare
Saldini 50, 20133 Milano, Italy}
\email{plankmathematics@gmail.com}

\date{\today}



\begin{document}

\begin{abstract}
    We study whether weak approximability glues along recollements of triangulated categories that are compactly generated by a single object.
    This is applied to derived categories of complexes with quasi-coherent cohomology and prescribed support for quasi-compact quasi-separated algebraic spaces.
\end{abstract}

\maketitle

\tableofcontents

\section{Introduction}
\label{sec:intro}

Approximable triangulated categories, and their weak variant, were introduced by Neeman \cite{Neeman:2021c,Neeman:2021,Neeman:2026}.
The theory has been used to settle conjectures of Bondal--Van den Bergh \cite{Neeman:2021} and of Antieau--Gepner--Heller \cite{Neeman:2024}.
It has also found applications to Fourier--Mukai theory for singular schemes  \cite{Dutta/Lank/ManaliRahul:2025,GuisadoVillaalgordo/Lank/ManaliRahul/Pavic:2025,GuisadoVillaalgordo/Lank:2026}.
See \cite{Canonaco/Neeman/Stellari:2024} for a survey of recent developments.

Roughly speaking, a triangulated category with small coproducts and a compact generator $G$ is weakly approximable if every object in the aisle of a suitable $t$-structure can be suitably approximated from small coproducts of a uniformly bounded range of shifts of $G$.
For example, if $X$ is a quasi-compact quasi-separated scheme and $Z\subseteq X$ is a closed subset with quasi-compact complement, $D_{\operatorname{qc},Z}(X)$ is weakly approximable \cite[Theorem 3.2(iv)]{Neeman:2024}.

The main result is the following. 

\begin{theorem}
    \label{thm:glueing_weak_approx}
    Consider a recollement of singly compactly generated triangulated categories:
    \begin{displaymath}
        \begin{tikzcd}[column sep=6em]
            {\mathcal{D}_Z} & {\mathcal{D}} & {\mathcal{D}_U.}
            \arrow["{i_\ast}" description, from=1-1, to=1-2]
            \arrow["{i^\ast}"', bend right=35, from=1-2, to=1-1]
            \arrow["{i^!}", bend left=35, from=1-2, to=1-1]
            \arrow["{j^\ast}" description, from=1-2, to=1-3]
            \arrow["{j_!}"', bend right=35, from=1-3, to=1-2]
            \arrow["{j_\ast}", bend left=35, from=1-3, to=1-2]
        \end{tikzcd}
    \end{displaymath}
    Assume that $\mathcal{D}_Z$ and $\mathcal{D}_U$ are weakly approximable. Suppose $G^\prime$ is a compact generator of $\mathcal{D}$.
    Then $\mathcal{D}$ is weakly approximable if, and only if, $\operatorname{Hom}(G^\prime,G^\prime[n])=0$ for $n\gg 0$.
\end{theorem}

The proof of \cite[Theorem 1.2]{DeDeyn/Lank/ManaliRahul/Peng:2025} does not rely on weak approximability of $D_{\operatorname{qc},Z}$ for algebraic stacks. 
As an application of \Cref{thm:glueing_weak_approx}, we establish weak approximability for algebraic spaces (see \Cref{rmk:stacks_version} for algebraic stacks). 

\begin{corollary}
    \label{cor:weak_approx_spaces}
    Let $X$ be a quasi-compact quasi-separated algebraic space.
    Suppose $S\subseteq |X|$ is a closed subset with quasi-compact complement.
    Then $D_{\operatorname{qc},S}(X)$ is weakly approximable. 
    In particular, $D_{\operatorname{qc}}(X)$ is weakly approximable.
    If $X$ is Noetherian, then $D_{\operatorname{qc},S}(X)$ is coherent.
\end{corollary}

\Cref{thm:glueing_weak_approx} is the analog for weak approximability of \cite[Theorem 5.1]{Burke/Neeman/Pauwels:2023} by Burke--Neeman--Pauwels.
Loc.\ cit.\ shows that $\mathcal{D}$ is approximable if $\mathcal{D}_Z$ and $\mathcal{D}_U$ are approximable and $\operatorname{Hom}(G^\prime,G^\prime[n])=0$ for $n\gg 0$.
The `only if' direction of \Cref{thm:glueing_weak_approx} is formal (cf.\ \cite[Remark 3]{Burke/Neeman/Pauwels:2023}).
Thus, the work lies in the `if' direction.

The proof in \cite{Burke/Neeman/Pauwels:2023} lifts approximations from $\mathcal{D}_Z$ to $\mathcal{D}$.
This is done by induction on the number of cones needed to build them \cite[Lemmas 5.2 to 5.5]{Burke/Neeman/Pauwels:2023}.
At each step, the bound on the lift is worse than on its target.
Hence, the constants grow with the number of cones.
Approximability of $\mathcal{D}_Z$ bounds this number.

No such bound need exist if $\mathcal{D}_Z$ is only weakly approximable.
Instead, we lift a single shift of a compact generator so that it lies in the same shifted aisle as its target (e.g.\ $\mathcal{D}^{\leq K}$).
See \Cref{lem:onecell}.
Consequently, the cone of a lift satisfies the same bound, which gives control on the desired objects.
See \Cref{prop:closure}.

A closure argument of this type also appears in unpublished work of Pauwels. 
See \Cref{rmk:weak_approx_spaces}.
Here the constants are explicit: if $A$ works for $\mathcal{D}_Z$, $\mathcal{D}_U$ and the vanishing $\operatorname{Hom}(G,\mathcal{D}^{\leq -A})=0$ for the compact generator $G$ of \Cref{lem:normalize}\eqref{lem:normalize2}, then $3A$ works for $\mathcal{D}$.

We discuss \Cref{cor:weak_approx_spaces}.
In general for schemes, $D_{\operatorname{qc},Z}$ need not be approximable \cite[Remark 8.1]{Neeman:2024}.
Hence, \cite[Theorem 5.1]{Burke/Neeman/Pauwels:2023} does not apply, and the weak analog is necessary.
The proof of \Cref{cor:weak_approx_spaces} combines Neeman's result for schemes \cite[Theorem 3.2(iv)]{Neeman:2024} with the induction principle for algebraic spaces \cite[\href{https://stacks.math.columbia.edu/tag/09IT}{Tag 09IT}]{StacksProject}.

\begin{ack}
    This question goes back to the author's residency at the Simons Laufer Mathematical Sciences Institute (formerly MSRI) in Spring 2024.
    The author thanks Fei Peng for helpful discussions, and thanks Bregje Pauwels for sharing her unpublished notes on this question (see \Cref{rmk:weak_approx_spaces}).
    The author was supported by the ERC Advanced Grant 101095900-TriCatApp.
\end{ack}

\paragraph*{\textbf{AI disclosure}}

The author used large language models (LLMs) for assistance with proofreading (e.g.\ grammar, notation, exposition) and checking proofs (e.g.\ finding gaps, pointing out errors).
The main results and methods are due to the author.

\section{Preliminaries}
\label{sec:prelim}

\subsection{Triangulated categories}
\label{sec:prelim_triangulated}

\subsubsection{Terminology}
\label{sec:prelim_terminology}

We fix some terminology. 
See \cite[\S 1]{Neeman:2001}.
Let $\mathcal{T}$ be a triangulated category.
Note that $\operatorname{Hom}(X,Y)$ is an abelian group.
Given a morphism $u$, write $-u$ for its additive inverse.
Consider a distinguished triangle
\begin{equation}
    \label{eq:triangle}
    X\xrightarrow{u} Y\xrightarrow{v} Z\xrightarrow{w} X[1].
\end{equation}
Then TR2 asserts that
\begin{displaymath}
    \begin{aligned}
        Y
        & \xrightarrow{v} Z\xrightarrow{w} X[1]\xrightarrow{-u[1]} Y[1],
        \\& Z[-1]\xrightarrow{-w[-1]} X\xrightarrow{u} Y\xrightarrow{v} Z
    \end{aligned}
\end{displaymath}
are also distinguished.
Passing from \eqref{eq:triangle} to the first (resp.\ second) of these is called \textbf{rotating} (resp.\ \textbf{rotating backwards}).
Observe that negating two of the three morphisms of a distinguished triangle gives a distinguished triangle. 
Indeed, $(1_X,-1_Y,1_Z)$ is an isomorphism of triangles from $(u,v,w)$ to $(-u,-v,w)$.
For every $n\in \mathbb{Z}$, there exists a distinguished triangle
\begin{displaymath}
    X[n]\xrightarrow{u[n]} Y[n]\xrightarrow{v[n]} Z[n]\xrightarrow{(-1)^n w[n]} X[n+1].
\end{displaymath}
We call this \textbf{shifting} the given distinguished triangle by $[n]$.
By TR1, every $f\colon X\to Y$ can be completed into a distinguished triangle 
\begin{displaymath}
    X\xrightarrow{f} Y\to C_f\to X[1].
\end{displaymath}
We call $C_f$ a \textbf{cone} of $f$.
Rotating backwards gives
\begin{displaymath}
    F_f\to X\xrightarrow{f} Y\to F_f[1].
\end{displaymath}
We call $F_f := C_f[-1]$ a \textbf{cocone} of $f$.
Both are unique up to (non-unique) isomorphism.
By combining TR1 and TR2, it follows that every morphism can be realized as the first, second, or third of some distinguished triangle.
Moreover, given a morphism of distinguished triangles, if two of the three vertical morphisms are isomorphisms, then so is the third.
Indeed, this can be shown by rotating and applying \cite[Proposition 1.1.20]{Neeman:2001}.

\subsubsection{Generation}
\label{sec:prelim_generation}

Let $\mathcal{T}$ be a triangulated category admitting small coproducts and $\mathcal{S},\mathcal{S}^\prime$ full subcategories.
Write $\mathcal{S} \star \mathcal{S}^\prime$ for the full subcategory of objects $X$ admitting a distinguished triangle $S\to X\to S^\prime \to S[1]$ with $S\in \mathcal{S}$ and $S^\prime \in \mathcal{S}^\prime$.
Write $\operatorname{Coprod}(\mathcal{S})$ for the smallest strictly full subcategory containing $\mathcal{S}$ that is closed under small coproducts and extensions, and $\operatorname{smd}(\mathcal{S})$ for the full subcategory of direct summands of objects of $\mathcal{S}$.
For $H\in \mathcal{T}$ and integers $a\leq b$ put
\begin{displaymath}
    H[a,b] := \{ H[-k] : a\leq k\leq b \},
    \qquad 
    \overline{\langle H\rangle}^{[a,b]} := \operatorname{smd}(\operatorname{Coprod}(H[a,b])).
\end{displaymath}

The subcategory of compact objects is denoted by $\mathcal{T}^c$, consisting of those $c \in \mathcal{T}$ for which $\operatorname{Hom}(c,-)$ preserves coproducts.
These form a triangulated subcategory of $\mathcal{T}$. 
The \textbf{localizing subcategory} generated by a collection $\mathcal{S}\subseteq \mathcal{T}$, denoted by $\overline{\langle \mathcal{S} \rangle}^{(-\infty,\infty)}$, is the smallest triangulated subcategory containing $\mathcal{S}$ and closed under small coproducts.
If $\mathcal{T}^c$ is essentially small, then we say that $\mathcal{T}$ is \textbf{compactly generated} when $\mathcal{T} = \overline{\langle \mathcal{T}^c \rangle}^{(-\infty,\infty)}$.
Equivalently, $\mathcal{T}$ is compactly generated if, for any $E \in \mathcal{T}$ satisfying $\operatorname{Hom}(P, E) = 0$ for all $P \in \mathcal{T}^c$, one has $E \cong 0$.
An object $G\in \mathcal{T}^c$ is a \textbf{compact generator} of $\mathcal{T}$ if every $E\in \mathcal{T}$ with $\operatorname{Hom}(G[n],E)=0$ for all $n\in \mathbb{Z}$ satisfies $E\cong 0$.
An object of $\mathcal{T}^c$ is a \textbf{classical generator} of $\mathcal{T}^c$ if the smallest thick subcategory of $\mathcal{T}^c$ containing it is $\mathcal{T}^c$.
If $\mathcal{T}$ is compactly generated, then classical generators for $\mathcal{T}^c$ coincide with compact generators for $\mathcal{T}$ \cite[\href{https://stacks.math.columbia.edu/tag/09SR}{Tag 09SR}]{StacksProject}.
We say that $\mathcal{T}$ is \textbf{singly compactly generated} if it admits a compact generator. 

\subsubsection{\texorpdfstring{$t$}{t}-structures}
\label{sec:prelim_t-structures_t_structures}

Let $\mathcal{T}$ be a triangulated category with shift functor $[1]$. 
A pair of strictly full subcategories $(\mathcal{T}^{\leq 0}, \mathcal{T}^{\geq 0})$ of $\mathcal{T}$ is a \textbf{$t$-structure} if:
\begin{itemize}
    \item $\operatorname{Hom}(A,B) = 0$ for all $A \in \mathcal{T}^{\leq 0}$ and $B \in \mathcal{T}^{\geq 0}[-1]$,
    \item $\mathcal{T}^{\leq 0}[1] \subseteq \mathcal{T}^{\leq 0}$ and $\mathcal{T}^{\geq 0}[-1] \subseteq \mathcal{T}^{\geq 0}$,
    \item for every $E \in \mathcal{T}$, there is a distinguished triangle
    \begin{displaymath}
        \tau^{\leq 0} E \to E \to \tau^{\geq 1} E \to (\tau^{\leq 0} E)[1]
    \end{displaymath}
    with $\tau^{\leq 0} E \in \mathcal{T}^{\leq 0}$ and $\tau^{\geq 1} E \in \mathcal{T}^{\geq 0}[-1]$.
\end{itemize}
This notion was introduced in \cite{Beilinson/Berstein/Deligne/Gabber:2018}. 
The above distinguished triangle is called the \textbf{truncation triangle} of $E$ with respect to the $t$-structure. 
For any $n \in \mathbb{Z}$, $(\mathcal{T}^{\leq n}, \mathcal{T}^{\geq n})$ is also a $t$-structure on $\mathcal{T}$ where $\mathcal{T}^{\leq n} := \mathcal{T}^{\leq 0}[-n]$ and $\mathcal{T}^{\geq n} := \mathcal{T}^{\geq 0}[-n]$. 
Define the subcategories:
\begin{displaymath}
    \mathcal{T}^{-}:= \bigcup_{n=0}^\infty \mathcal{T}^{\leq n},\quad \mathcal{T}^{+} := \bigcup^\infty_{n=0} \mathcal{T}^{\geq -n},\quad \mathcal{T}^b := \mathcal{T}^{-} \cap \mathcal{T}^{+}.
\end{displaymath}
Let $F \colon \mathcal{T} \to \mathcal{T}^\prime$ be an exact functor between triangulated categories equipped with $t$-structures $(\mathcal{T}^{\leq 0}, \mathcal{T}^{\geq 0})$ and $((\mathcal{T}^\prime)^{\leq 0}, (\mathcal{T}^\prime)^{\geq 0})$. 
We say that $F$ is \textbf{right $t$-exact} if $
F(\mathcal{T}^{\leq 0}) \subseteq (\mathcal{T}^\prime)^{\leq 0}$,
and \textbf{left $t$-exact} if 
$F(\mathcal{T}^{\geq 0}) \subseteq (\mathcal{T}^\prime)^{\geq 0}$. 
If both conditions hold, then $F$ is \textbf{$t$-exact}.

A strictly full subcategory $\mathcal{A} \subseteq \mathcal{T}$ is an \textbf{aisle} if it is closed under positive shifts and extensions, and the inclusion $\mathcal{A} \to \mathcal{T}$ admits a right adjoint. 
Given a $t$-structure $(\mathcal{T}^{\leq 0}, \mathcal{T}^{\geq 0})$, the subcategory $\mathcal{T}^{\leq 0}$ is an aisle. 
By \cite[\S 1, Proposition]{Keller/Vossieck:1988}, any aisle $\mathcal{A}$ determines a $t$-structure $(\mathcal{A}, \mathcal{A}^\perp[1])$, where
\begin{displaymath}
    \mathcal{A}^\perp := \{ T \in \mathcal{T} : \operatorname{Hom}(A,T) = 0 \text{ for all } A \in \mathcal{A} \}.
\end{displaymath}
The subcategories $\mathcal{T}^{\leq 0}$ and $\mathcal{T}^{\geq 0}$ are called the \textbf{aisle} and \textbf{coaisle} of the $t$-structure, respectively. 
More generally, a \textbf{preaisle} is a strictly full subcategory closed under positive shifts and extensions. 
By \cite[Lemma 1.3 \& Corollary 1.4]{AlonsoTarrio/Lopez/Salorio:2003}, aisles are closed under direct summands and under small coproducts when these exist in $\mathcal{T}$.

Assume $\mathcal{T}$ admits small coproducts. 
A preaisle $\mathcal{A} \subseteq \mathcal{T}$ is called \textbf{cocomplete} if it is closed under all coproducts in $\mathcal{T}$. 
Given $\mathcal{S}\subseteq \mathcal{T}$, $\overline{\langle \mathcal{S} \rangle}^{(-\infty, 0]}$ is defined to be the smallest cocomplete preaisle containing $\mathcal{S}$. 
By \cite[Theorem 2.3]{Neeman:2021b}, if $\mathcal{T}$ is well generated (e.g.\ compactly generated), $\overline{\langle \mathcal{S} \rangle}^{(-\infty, 0]}$ is an aisle whenever $\mathcal{S}$ is essentially small. 
Furthermore, if $\mathcal{A}\subseteq \mathcal{T}^c$ is closed under positive shifts, then \cite[Theorem 2.3.3 \& Remark 2.3.4]{Canonaco/Haesemeyer/Neeman/Stellari:2024} (which generalizes \cite[Theorem A.1 \& Proposition A.2]{AlonsoTarrio/Lopez/Salorio:2003}) says that $\operatorname{Coprod}(\mathcal{A})$ is the aisle $\overline{\langle \mathcal{A} \rangle}^{(-\infty, 0]}$. 
An aisle $\mathcal{U}$ on $\mathcal{T}$ is \textbf{compactly generated} when there exists a collection of compact objects $\mathcal{Q} \subseteq \mathcal{T}^c$ satisfying $\overline{\langle \mathcal{Q} \rangle}^{(-\infty, 0]} = \mathcal{U}$. 
In this case, we say $\mathcal{U}$ is compactly generated by $\mathcal{Q}$. 
We say a $t$-structure is \textbf{compactly generated} if its aisle is as such. 
For instance, if $G \in \mathcal{T}^c$, then 
\begin{displaymath}
    \operatorname{Coprod}( \{ G[n] : n\geq 0 \}) = \overline{\langle G \rangle}^{(-\infty, 0]} .
\end{displaymath}

Consider a pair of $t$-structures $(\mathcal{T}_1^{\leq 0},\mathcal{T}_1^{\geq 0})$ and $(\mathcal{T}_2^{\leq 0},\mathcal{T}_2^{\geq 0})$ on $\mathcal{T}$.
We say they are \textbf{equivalent} if there exists $A\geq 0$ such that $\mathcal{T}_1^{\leq -A}\subseteq \mathcal{T}_2^{\leq 0}\subseteq \mathcal{T}_1^{\leq A}$.
If $\mathcal{T}$ has a single compact generator, the \textbf{preferred equivalence class} is the equivalence class of the $t$-structure compactly generated by any compact generator.
By \cite[Remark 1.19]{Neeman:2026}, the preferred equivalence class does not depend on the choice of compact generator. 

If $\mathcal{T}$ has small coproducts and is equipped with a $t$-structure $(\mathcal{T}^{\leq 0}, \mathcal{T}^{\geq 0})$, let $\mathcal{T}^-_c$ be the collection of $E\in \mathcal{T}$ such that for every $n>0$ there exists a distinguished triangle 
\begin{displaymath}
    D \to E \to F \to D[1]
\end{displaymath}
with $D\in \mathcal{T}^c$ and $F\in \mathcal{T}^{\leq -n}$. 
Moreover, let $\mathcal{T}^b_c := \mathcal{T}^b \cap \mathcal{T}^-_c$. 
Furthermore, if $\mathcal{T}$ is weakly approximable and $(\mathcal{T}^{\leq 0}, \mathcal{T}^{\geq 0})$ is in the preferred equivalence class, then $\mathcal{T}^?$ and $\mathcal{T}^\#_c$ are independent of the $t$-structure representing the equivalence class for $?\in \{-,+,b\}$ and $\# \in \{-,b\}$. 
See \cite[Remark 4.1.5]{Canonaco/Haesemeyer/Neeman/Stellari:2024}.

\subsubsection{Recollements}
\label{sec:prelim_recollements}

We briefly recall the notion of a recollement. 
See \cite[\S 1.4]{Beilinson/Berstein/Deligne/Gabber:2018} for details. 
A \textbf{recollement} is a diagram of triangulated categories and exact functors of the form 
\begin{displaymath}
        \begin{tikzcd}[column sep=6em]
            {\mathcal{D}_Z} & {\mathcal{D}} & {\mathcal{D}_U}
            \arrow["{i_\ast}" description, from=1-1, to=1-2]
            \arrow["{i^\ast}"', bend right=35, from=1-2, to=1-1]
            \arrow["{i^!}", bend left=35, from=1-2, to=1-1]
            \arrow["{j^\ast}" description, from=1-2, to=1-3]
            \arrow["{j_!}"', bend right=35, from=1-3, to=1-2]
            \arrow["{j_\ast}", bend left=35, from=1-3, to=1-2]
        \end{tikzcd}
    \end{displaymath}
satisfying:
\begin{itemize}
    \item $i^\ast \dashv i_\ast \dashv i^!$ and $j_! \dashv j^\ast \dashv j_\ast$ (i.e.\ adjoint triples)
    \item $i_\ast, j_!, j_\ast$ are fully faithful
    \item $\ker (j^\ast)$ coincides with the strictly full subcategory on objects of the form $i_\ast(T)$ where $T\in \mathcal{D}_Z$.
\end{itemize}
In such a case, there are distinguished triangles
\begin{displaymath}
    \begin{aligned}
        j_! j^\ast E 
        & \xrightarrow{\epsilon_E} E 
        \xrightarrow{\eta_E} i_\ast i^\ast E 
        \xrightarrow{\varphi_E} j_! j^\ast E[1],
        \\& i_\ast i^! E \to E \to j_\ast j^\ast E \to i_\ast i^! E[1]
    \end{aligned}
\end{displaymath}
which are functorial in $E\in \mathcal{D}$.
In particular, the natural transformations between these functors are given by the (co)units of the relevant adjoint pairs. 
Since $j_!$, $j^\ast$, $i_\ast$, and $i^\ast$ are left adjoints, they preserve coproducts. 
Additionally, if each triangulated category is compactly generated, $i_\ast$ and $j^\ast$ admitting right adjoints implies $i^\ast$ and $j_!$ preserve compact objects \cite[Theorem 5.1]{Neeman:1996}.
Given $E\in \mathcal{D}_Z$ and $B \in \mathcal{D}_U$, $j^\ast i_\ast E \cong 0$, $i^! j_\ast B \cong 0$, and $i^\ast j_! B \cong 0$. 

Assume $(\mathcal{D}_Z^{\leq 0}, \mathcal{D}_Z^{\geq 0})$ and $(\mathcal{D}_U^{\leq 0}, \mathcal{D}_U^{\geq 0})$ are $t$-structures respectively on $\mathcal{D}_Z$ and $\mathcal{D}_U$. 
It is possible to `glue' these aisles to obtain a new $t$-structure on $\mathcal{D}$. 
Specifically, the glued $t$-structure is given by 
\begin{displaymath}
    \begin{aligned}
        \mathcal{D}^{\leq 0} &\colonequals  \left\{ E\in \mathcal{D} : i^\ast E \in \mathcal{D}_Z^{\leq 0}, j^\ast E \in \mathcal{D}_U^{\leq 0} \right\},
        \\ \mathcal{D}^{\geq 0} &\colonequals  \left\{ E\in \mathcal{D} : i^! E \in \mathcal{D}_Z^{\geq 0}, j^\ast E \in \mathcal{D}_U^{\geq 0} \right\}.
    \end{aligned}
\end{displaymath}
Moreover, a few general facts include: $i^\ast$ and $j_!$ are right $t$-exact;  $i_\ast$ and $j^\ast$ are $t$-exact; and $i^!$ and $j_\ast$ are left $t$-exact.
See \cite[1.4.10 \& 1.3.17.(iii)]{Beilinson/Berstein/Deligne/Gabber:2018} for details.

\subsubsection{Weak approximability}
\label{sec:prelim_weak_approx}

We recall the definition of weak approximability for triangulated categories.
See \cite[Definition 1.25]{Neeman:2026}.

\begin{definition}
    \label{def:weak_approx}
    A triangulated category $\mathcal{T}$ is \textbf{weakly approximable} if there exist a compact generator $G$, a $t$-structure $(\mathcal{T}^{\leq 0},\mathcal{T}^{\geq 0})$ and an integer $A>0$ such that:
    \begin{enumerate}
        \item \label{def:weak_approx1} $G\in \mathcal{T}^{\leq A}$ and $\operatorname{Hom}(G,\mathcal{T}^{\leq -A})=0$;
        \item \label{def:weak_approx2} every $F\in \mathcal{T}^{\leq 0}$ admits a distinguished triangle $E\to F\to D\to E[1]$ with $E\in \overline{\langle G\rangle}^{[-A,A]}$ and $D\in \mathcal{T}^{\leq -1}$.
    \end{enumerate}
    We call $(G,\mathcal{T}^{\leq 0},A)$ \textbf{weak approximation data}.
\end{definition}

In our convention, compactly generated triangulated categories a priori have small coproducts.
In \Cref{def:weak_approx}, \eqref{def:weak_approx1} is the condition $G[A]\in \mathcal{T}^{\leq 0}$, $\operatorname{Hom}(G[-A],\mathcal{T}^{\leq 0})=0$ of \cite[2.9]{Burke/Neeman/Pauwels:2023}, since $\operatorname{Hom}(G[-A],\mathcal{T}^{\leq 0}) = \operatorname{Hom}(G,\mathcal{T}^{\leq -A})$.

A weakly approximable triangulated category $\mathcal{T}$ is \textbf{coherent} if, for any $t$-structure $(\mathcal{T}^{\leq 0}, \mathcal{T}^{\geq 0})$ in the preferred equivalence class, there exists an integer $N > 0$ such that every object $Y \in \mathcal{T}^-_c$ admits a distinguished triangle 
\begin{displaymath}
    X \to Y \to Z \to X[1]
\end{displaymath}
with $X \in \mathcal{T}^-_c \cap \mathcal{T}^{\leq N}$ and $Z\in \mathcal{T}^b_c \cap \mathcal{T}^{\geq 0}$.
See \cite[Definition 10.1.1]{Canonaco/Haesemeyer/Neeman/Stellari:2024}.

\subsection{Algebraic spaces}
\label{sec:preliminaries_stacks}

We follow \cite{StacksProject}. 
Throughout, we work over a fixed base scheme. 
Let $X$ be a quasi-compact quasi-separated algebraic space. 

\subsubsection{Notions}
\label{sec:prelim_notions}

An \textbf{\'{e}tale presentation} of $X$ is an \'{e}tale, finitely presented, surjective morphism to $X$ from a scheme. 
Denote by $|X|$ the underlying topological space of $X$. 
It consists of equivalence classes of morphisms from fields to $X$ \cite[\href{https://stacks.math.columbia.edu/tag/03BT}{Tag 03BT}]{StacksProject}. 

\subsubsection{Categories \& functors}
\label{sec:prelim_categories}

$\operatorname{Mod}(X)$ is the Grothendieck abelian category of sheaves of $\mathcal{O}_X$-modules on the small \'{e}tale site $X_{\textrm{\'{e}tale}}$ of $X$ \cite[\href{https://stacks.math.columbia.edu/tag/03LT}{Tag 03LT}]{StacksProject}. 
$\operatorname{Qcoh}(X)$ is the full subcategory of $\operatorname{Mod}(X)$ consisting of quasi-coherent sheaves. 
$D(X) := D(\operatorname{Mod}(X))$ is the derived category of $\operatorname{Mod}(X)$. 
$D_{\operatorname{qc}}(X)$ is the full subcategory of $D(X)$ consisting of complexes with quasi-coherent cohomology sheaves. 
$\operatorname{Perf}(X)$ is the full subcategory of perfect complexes in $D_{\operatorname{qc}}(X)$. 
If $X$ is Noetherian, then $\operatorname{coh}(X)$ is the full subcategory of $\operatorname{Mod}(X)$ consisting of coherent sheaves, and $D^?_{\operatorname{coh}}(X)$ is the full subcategory of $D(X)$ consisting of complexes with coherent cohomology sheaves with familiar decorations $?\in \{-,b,+\}$.

Let $f\colon Y \to X$ be a morphism of quasi-compact quasi-separated algebraic spaces. 
By \cite[\href{https://stacks.math.columbia.edu/tag/07A6}{Tag 07A6}]{StacksProject}, there exists an adjoint pair 
\begin{displaymath}
    \mathbf{L}f^\ast \colon D(X) \leftrightarrows D(Y) \colon \mathbf{R}f_\ast.
\end{displaymath}
By \cite[\href{https://stacks.math.columbia.edu/tag/08FA}{Tags 08FA} \& \href{https://stacks.math.columbia.edu/tag/08F4}{08F4}]{StacksProject}, $\mathbf{L}f^\ast$ and $\mathbf{R}f_\ast$ preserve quasi-coherent cohomology.
Moreover, \cite[\href{https://stacks.math.columbia.edu/tag/07A6}{Tag 07A6}]{StacksProject} says the restrictions of $\mathbf{L}f^\ast$ and $\mathbf{R}f_\ast$ form an adjoint pair on $D_{\operatorname{qc}}$. 

\subsubsection{Support}
\label{sec:preliminaries_stacks_support}

For $M \in \operatorname{Qcoh}(X)$, $\operatorname{Supp}(M) :=|p|(\operatorname{Supp}(p^\ast M)) \subseteq |X|$ where $p \colon U \to X$ is any \'{e}tale presentation.
By \cite[\href{https://stacks.math.columbia.edu/tag/07TY}{Tag 07TY}]{StacksProject}, $\operatorname{Supp}$ is independent of the \'{e}tale presentation. 
It coincides with the notion of support for abelian sheaves on the small \'{e}tale site \cite[\href{https://stacks.math.columbia.edu/tag/04K7}{Tag 04K7}]{StacksProject}.
More generally, if $E \in D_{\operatorname{qc}}(X)$, 
\begin{displaymath}
    \operatorname{Supp}(E) := \bigcup_{j \in \mathbb{Z}} \operatorname{Supp}(\mathcal{H}^j(E)) \subseteq |X|.
\end{displaymath}
Given a closed subset $Z \subseteq |X|$, $E\in D_{\operatorname{qc}}(X)$ is \textbf{supported on $Z$} if $\operatorname{Supp}(E) \subseteq Z$. 
Denote by $D_{\operatorname{qc},Z}(X)$ the strictly full subcategory of $D_{\operatorname{qc}}(X)$ consisting of objects supported on $Z$. 
If $X$ is Noetherian, we can define $D_{\operatorname{coh},Z}(X)$ analogously. 

\subsubsection{Perfect complexes}
\label{sec:prelim_perfects}

Perfect complexes are defined on the small \'{e}tale site of $X$. 
See \cite[\href{https://stacks.math.columbia.edu/tag/08G4}{Tag 08G4}]{StacksProject}. 
In particular, a complex is \textbf{strictly perfect} if it is a bounded complex whose terms are direct summands of finite free modules, whereas a complex is \textbf{perfect} if it is locally strictly perfect. 
For any $Z\subseteq |X|$ closed such that $|X|\setminus Z$ is quasi-compact: perfect complexes supported on $Z$ coincide with compacts of $D_{\operatorname{qc},Z}(X)$ and $D_{\operatorname{qc},Z}(X)$ is singly compactly generated \cite[\href{https://stacks.math.columbia.edu/tag/0AED}{Tags 0AED}, \href{https://stacks.math.columbia.edu/tag/09M8}{09M8}, \& \href{https://stacks.math.columbia.edu/tag/0AEC}{0AEC}]{StacksProject}.

\section{Auxiliary results}
\label{sec:auxillary}

\begin{lemma}
    \label{lem:closures}
    Let $\mathcal{T}$ and $\mathcal{T}^\prime$ be triangulated categories which admit small coproducts. 
    Choose $H\in \mathcal{T}$ and integers $a\leq b$.
    Then:
    \begin{enumerate}
        \item \label{lem:closures1} $\overline{\langle H\rangle}^{[a,b]}$ is closed under small coproducts, extensions, and direct summands; in fact, it is the smallest strictly full subcategory of $\mathcal{T}$ closed under these operations that contains $H[a,b]$.
        \item \label{lem:closures2} $\overline{\langle H\rangle}^{[a,b]}[s] = \overline{\langle H\rangle}^{[a-s,b-s]}$ and $\overline{\langle H\rangle}^{[a,b]} \subseteq \overline{\langle H\rangle}^{[a^\prime,b^\prime]}$ for all $s\in \mathbb{Z}$, $a^\prime\leq a$, and $b\leq b^\prime$.
        \item \label{lem:closures3} $F(\overline{\langle H\rangle}^{[a,b]}) \subseteq \overline{\langle H^\prime\rangle}^{[a+c,b+d]}$ for any small coproduct preserving exact $F\colon \mathcal{T}\to \mathcal{T}^\prime$ and $H^\prime \in \mathcal{T}^\prime$ with $F(H)\in \overline{\langle H^\prime\rangle}^{[c,d]}$.
        \item \label{lem:closures4} Let $Y\in \mathcal{T}$. If $\operatorname{Hom}(H[-k],Y)=0$ for all $a\leq k\leq b$, then $\operatorname{Hom}(X,Y)=0$ for all $X\in \overline{\langle H\rangle}^{[a,b]}$. If $H$ is compact, then $\{Y\in \mathcal{T} : \operatorname{Hom}(H,Y)=0\}$ is closed under small coproducts, extensions, and direct summands.
        \item \label{lem:closures5} $\operatorname{smd}(\operatorname{Coprod}(H[a,b])) \subseteq \operatorname{Coprod}(H[a,b+1])$.
    \end{enumerate}
\end{lemma}

\begin{proof}
    First we prove \eqref{lem:closures1}.
    By definition, $\overline{\langle H\rangle}^{[a,b]} := \operatorname{smd}(\operatorname{Coprod}(H[a,b]))$. 
    Let $\Lambda$ be a set.
    Choose $X_\lambda \in \overline{\langle H\rangle}^{[a,b]}$ where $\lambda \in \Lambda$. 
    For each $\lambda \in \Lambda$, there exists $X^\prime_\lambda \in \mathcal{T}$ such that $X_\lambda \oplus X^\prime_\lambda\in \operatorname{Coprod}(H[a,b])$.
    Since $\operatorname{Coprod}(H[a,b])$ is closed under small coproducts, $\oplus_{\lambda\in \Lambda} (X_\lambda \oplus X^\prime_\lambda)\in \operatorname{Coprod}(H[a,b])$. 
    Note that $\oplus_{\lambda\in \Lambda}X_\lambda$ is a direct summand of $\oplus_{\lambda\in \Lambda} (X_\lambda \oplus X^\prime_\lambda)$. 
    Hence, $\overline{\langle H\rangle}^{[a,b]}$ is closed under small coproducts. 

    Consider a distinguished triangle 
    \begin{displaymath}
        A \to B \to C \to A[1]
    \end{displaymath}
    where $A,C\in \overline{\langle H\rangle}^{[a,b]}$. 
    We can find $A^\prime,C^\prime$ such that $A\oplus A^\prime, C\oplus C^\prime \in \operatorname{Coprod}(H[a,b])$. 
    There exist distinguished triangles
    \begin{displaymath}
        A^\prime \xrightarrow{1_{A^\prime}} A^\prime \to 0 \to A^\prime[1]
    \end{displaymath}
    and \begin{displaymath}
        0 \to C^\prime \xrightarrow{1_{C^\prime}} C^\prime \to 0[1].
    \end{displaymath}
    Taking the direct sum of these distinguished triangles yields a distinguished triangle 
    \begin{displaymath}
        \begin{aligned}
            A\bigoplus A^\prime \bigoplus 0 
            \to B \bigoplus A^\prime \bigoplus C^\prime
            \to C \bigoplus 0 \bigoplus C^\prime \to A[1]\bigoplus A^\prime[1] \bigoplus 0.
        \end{aligned}
    \end{displaymath}
    See \cite[Proposition 1.2.1]{Neeman:2001}. 
    Since $A\oplus A^\prime, C\oplus C^\prime\in  \operatorname{Coprod}(H[a,b])$, it follows that $B \oplus A^\prime \oplus C^\prime\in \operatorname{Coprod}(H[a,b])$ because the category is closed under extensions.
    Consequently, $B \in \overline{\langle H\rangle}^{[a,b]}$.

    To see that $\overline{\langle H\rangle}^{[a,b]}$ is closed under direct summands, use the fact that a direct summand of a direct summand remains a direct summand.

    Lastly, we show the minimality claim. 
    Let $\mathcal{B}$ be a strictly full subcategory of $\mathcal{T}$ closed under small coproducts, extensions, and direct summands that contains $H[a,b]$.
    In this case, $\operatorname{Coprod}(H[a,b])\subseteq \mathcal{B}$. 
    Indeed, $\mathcal{B}$ is closed under small coproducts and extensions, and it contains $H[a,b]$. 
    Moreover, $\operatorname{smd}(\operatorname{Coprod}(H[a,b]))\subseteq \mathcal{B}$ because $\mathcal{B}$ is closed under direct summands. 

    Second we prove \eqref{lem:closures2}. 
    The autoequivalence $[s]\colon \mathcal{T} \to \mathcal{T}$ preserves extensions, small coproducts, and direct summands. 
    Hence, $\overline{\langle H\rangle}^{[a,b]}[s]$ is the smallest strictly full subcategory containing $(H[a,b])[s]$ which is closed under extensions, small coproducts, and direct summands.
    However, shifting by $s$ yields 
    \begin{displaymath}
        \begin{aligned}
            (H[a,b])[s] 
            &= \{ H [-k][s] : a\leq k \leq b \} 
            \\&= \{ H [-k+s] : a\leq k \leq b \}
            \\&=\{ H [-l] : a-s\leq l \leq b-s \}
            \\&=: H[a-s,b-s].
        \end{aligned}
    \end{displaymath}
    It follows that $\overline{\langle H\rangle}^{[a,b]}[s] = \overline{\langle H\rangle}^{[a-s,b-s]}$. 
    Moreover, as $H[a,b]\subseteq H[a^\prime,b^\prime]$, \eqref{lem:closures1} implies $\overline{\langle H\rangle}^{[a,b]} \subseteq \overline{\langle H\rangle}^{[a^\prime,b^\prime]}$.

    Third we prove \eqref{lem:closures3}.
    Choose $a\leq k\leq b$.
    Note that $F$ is exact and $F(H)\in \overline{\langle H^\prime\rangle}^{[c,d]}$.
    Hence, by \eqref{lem:closures2},
    \begin{displaymath}
        F(H[-k])\cong F(H)[-k]
        \in \overline{\langle H^\prime\rangle}^{[c,d]}[-k] 
        = \overline{\langle H^\prime\rangle}^{[c+k,d+k]}\subseteq \overline{\langle H^\prime\rangle}^{[a+c,b+d]}.
    \end{displaymath}
    Set $\mathcal{S}$ to be the full subcategory of $E\in \mathcal{T}$ such that $F(E)\in \overline{\langle H^\prime\rangle}^{[a+c,b+d]}$.
    It follows that $H[a,b]\subseteq \mathcal{S}$.
    By \eqref{lem:closures1}, $\overline{\langle H^\prime\rangle}^{[a+c,b+d]}$ is strictly full and closed under small coproducts, extensions, and direct summands.
    Since $F$ preserves small coproducts, $\mathcal{S}$ has the same properties.
    By minimality, \eqref{lem:closures1} implies $\overline{\langle H\rangle}^{[a,b]}\subseteq \mathcal{S}$.
    Thus, \eqref{lem:closures3} follows.

    Next we prove \eqref{lem:closures4}. 
    It is easy to see that $\{E : \operatorname{Hom}(E,Y)=0\}$ is strictly full. 
    Let $\Lambda$ be a set. 
    Choose $X_\lambda \in \{E : \operatorname{Hom}(E,Y)=0\}$.
    By the universal property of coproducts, there exists an isomorphism
    \begin{displaymath}
        \operatorname{Hom}(\bigoplus_{\lambda \in \Lambda} X_\lambda, Y) 
        \cong \prod_{\lambda \in \Lambda} \operatorname{Hom}(X_\lambda, Y)
        \cong 0.
    \end{displaymath} 
    See \cite[Example 1.1.13]{Neeman:2001} applied to $\mathcal{T}^{op}$. 
    The last isomorphism follows from the fact that each $X_\lambda \in \{E : \operatorname{Hom}(E,Y)=0\}$.
    Hence, $\{E : \operatorname{Hom}(E,Y)=0\}$ is closed under small coproducts. 
    Let $A,B\in \{E : \operatorname{Hom}(E,Y)=0\}$.
    Consider an extension $C$ given by a distinguished triangle 
    \begin{displaymath}
        A \to C \to B \to A[1].
    \end{displaymath}
    Apply $\operatorname{Hom}(-,Y)$ to obtain a long exact sequence in Hom sets 
    \begin{displaymath}
        \cdots \to \operatorname{Hom}(B,Y) \to \operatorname{Hom}(C,Y) \to \operatorname{Hom}(A,Y) \to \cdots.
    \end{displaymath}
    Hence, $\operatorname{Hom}(C,Y) \cong 0$ if 
    \begin{displaymath}
        \operatorname{Hom}(A,Y) \cong \operatorname{Hom}(B,Y) \cong 0.
    \end{displaymath}
    This implies $\{E : \operatorname{Hom}(E,Y)=0\}$ is closed under extensions.
    Moreover, if $A\oplus B \in \{E : \operatorname{Hom}(E,Y)=0\}$, then 
    \begin{displaymath}
        \operatorname{Hom}(A\oplus B,Y)\cong \operatorname{Hom}(A,Y) \bigoplus \operatorname{Hom}(B,Y) \cong 0.
    \end{displaymath} 
    Thus, $\{E : \operatorname{Hom}(E,Y)=0\}$ is closed under direct summands. 
    The hypothesis asserts that $H[a,b]\subseteq \{E : \operatorname{Hom}(E,Y)=0\}$.
    By \eqref{lem:closures1}, $\overline{\langle H \rangle}^{[a,b]}\subseteq \{E : \operatorname{Hom}(E,Y)=0\}$.
    A similar argument with $H$ compact proves the second claim for \eqref{lem:closures4}.
    Indeed, the compactness of $H$ gives the closure under coproducts, whereas the other conditions are analogous.
    
    Lastly we prove \eqref{lem:closures5}. 
    Let $X\in \operatorname{smd}(\operatorname{Coprod}(H[a,b]))$.
    Choose $X^\prime$ such that $X\oplus X^\prime \in \operatorname{Coprod}(H[a,b])$.
    The so-called Eilenberg swindle yields an isomorphism
    \begin{displaymath}
        \begin{aligned}
            X \oplus (\bigoplus_{n\in \mathbb{N}} (X\oplus X^\prime)) 
            &\cong X \oplus (\bigoplus_{n\in \mathbb{N}} (X^\prime \oplus X)) 
            \\&\cong X \oplus (X^\prime \oplus X) \oplus (X^\prime \oplus X) \oplus \cdots
            \\&\cong (X \oplus X^\prime) \oplus (X \oplus X^\prime) \oplus \cdots
            \\&\cong \bigoplus_{n\in \mathbb{N}} (X \oplus X^\prime).
        \end{aligned}
    \end{displaymath}
    Consequently, there exists a distinguished triangle 
    \begin{displaymath}
        X \to X \oplus (\bigoplus_{n\in \mathbb{N}} (X\oplus X^\prime)) \to \bigoplus_{n\in \mathbb{N}}(X\oplus X^\prime) \to X[1].
    \end{displaymath}
    Rotating backwards yields 
    \begin{displaymath}
        (\bigoplus_{n\in \mathbb{N}}( X\oplus X^\prime))[-1] 
        \to X \to X \oplus (\bigoplus_{n\in \mathbb{N}} (X\oplus X^\prime)) \to \bigoplus_{n\in \mathbb{N}}( X\oplus X^\prime).
    \end{displaymath}
    Since $[a,b]\subseteq [a,b+1]$, $\operatorname{Coprod}(H[a,b]) \subseteq \operatorname{Coprod}(H[a,b+1])$.
    Observe that 
    \begin{displaymath}
        \begin{aligned}
            (\bigoplus_{n\in \mathbb{N}} (X\oplus X^\prime))[-1]
            &\in \operatorname{Coprod}(H[a,b])[-1] 
            \\&= \operatorname{Coprod}(H[a+1,b+ 1])
            \\&\subseteq \operatorname{Coprod}(H[a,b+1]).
        \end{aligned}
    \end{displaymath}
    Therefore, 
    \begin{displaymath}
        \begin{aligned}
            X 
            &\in \operatorname{Coprod}(H[a,b+1]) \star \operatorname{Coprod}(H[a,b])
            \\&\subseteq \operatorname{Coprod}(H[a,b+1]) \star \operatorname{Coprod}(H[a,b+1])
            \\&\subseteq \operatorname{Coprod}(H[a,b+1]).
        \end{aligned}
    \end{displaymath}
    This completes the proof.
\end{proof}

\begin{remark}
    \label{rmk:octahedral}
    We use the octahedral axiom in the following ways.
    Let $X\xrightarrow{f} Y\xrightarrow{g} Z$ be composable.
    \begin{enumerate}
        \item \label{rmk:octahedraloc} Choose distinguished triangles $X\xrightarrow{f} Y\to C_f\to X[1]$, $Y\xrightarrow{g} Z\xrightarrow{p_g} C_g\to Y[1]$ and $X\xrightarrow{gf} Z\xrightarrow{p_{gf}} C_{gf}\to X[1]$. There is a distinguished triangle $C_f\to C_{gf}\xrightarrow{h} C_g\to C_f[1]$ with $h\circ p_{gf} = p_g$. See e.g.\ \cite[Proposition 1.4.6]{Neeman:2001}.
        \item \label{rmk:octahedralof} Choose distinguished triangles $F_f\to X\xrightarrow{f} Y\to F_f[1]$, $F_g\to Y\xrightarrow{g} Z\to F_g[1]$ and $F_{gf}\to X\xrightarrow{gf} Z\to F_{gf}[1]$. There is a distinguished triangle $F_f\to F_{gf}\to F_g\to F_f[1]$. Indeed, take the distinguished triangle $C_f\to C_{gf}\xrightarrow{h} C_g\to C_f[1]$ and shift by $[-1]$. Here we use that $F_? \cong C_?[-1]$.
    \end{enumerate}
\end{remark}

\begin{lemma}
    \label{lem:aisle}
    Let $(\mathcal{T}^{\leq 0},\mathcal{T}^{\geq 0})$ be a $t$-structure on a triangulated category $\mathcal{T}$.
    Consider a distinguished triangle  
    \begin{displaymath}
        E\to F\to D\to E[1]
    \end{displaymath}
    where $F\in \mathcal{T}^{\leq 0}$ and $D\in \mathcal{T}^{\leq -m}$ for some $m\geq 1$.
    Then $E\in \mathcal{T}^{\leq 0}$.
\end{lemma}

\begin{proof}
    By rotating backwards, there exists a distinguished triangle  
    \begin{displaymath}
       D[-1] \to E\to F\to D.
    \end{displaymath}
    Note that $D[-1]\in \mathcal{T}^{\leq 1-m} \subseteq \mathcal{T}^{\leq 0}$.
    Since aisles are closed under extensions, $E\in \mathcal{T}^{\leq 0}$.
\end{proof}

\begin{lemma}
    \label{lem:iterate}
    Let $(H,\mathcal{T}^{\leq 0},A)$ be weak approximation data for $\mathcal{T}$.
    For every $F\in \mathcal{T}^{\leq 0}$ and every $m\geq 1$, there exists a distinguished triangle 
    \begin{displaymath}
        E_m\to F\to D_m\to E_m[1]
    \end{displaymath}
    with $E_m\in \overline{\langle H\rangle}^{[1-m-A,A]} \cap \mathcal{T}^{\leq 0}$ and $D_m\in \mathcal{T}^{\leq -m}$. 
\end{lemma}

\begin{proof}
    This follows from \cite[Corollary 3.2(1)]{Neeman:2026} with weak approximation data for $\mathcal{T}$.
    To be self-contained, we give a proof.
    By \Cref{lem:aisle}, for any desired distinguished triangle 
    \begin{displaymath}
        E_m\to F\to D_m\to E_m[1],
    \end{displaymath}
    it follows that $E_m\in \mathcal{T}^{\leq 0}$.
    Hence, it suffices to produce the desired distinguished triangle.
    We argue by induction on $m$. 
    The base case $m=1$ is \Cref{def:weak_approx}\eqref{def:weak_approx2}.

    Assume that the desired distinguished triangle exists for $m$, i.e.\ a distinguished triangle 
    \begin{displaymath}
        E_m\to F\xrightarrow{f} D_m\to E_m[1]
    \end{displaymath}
    with $E_m\in \overline{\langle H\rangle}^{[1-m-A,A]} \cap \mathcal{T}^{\leq 0}$ and $D_m\in \mathcal{T}^{\leq -m}$. 
    Since $D_m[-m]\in \mathcal{T}^{\leq 0}$, \Cref{def:weak_approx}\eqref{def:weak_approx2} yields a distinguished triangle 
    \begin{displaymath}
        E^\prime \to D_m[-m]\to D^\prime \to E^\prime[1]
    \end{displaymath}
    where $E^\prime \in \overline{\langle H\rangle}^{[-A,A]}$ and $D^\prime \in \mathcal{T}^{\leq -1}$.
    After shifting by $[m]$, we have a distinguished triangle 
    \begin{displaymath}
        E^\prime[m]\to D_m\xrightarrow{g} D^\prime[m]\to E^\prime[m+1]
    \end{displaymath}
    where $E^\prime[m]\in \overline{\langle H\rangle}^{[-A-m,A-m]}$ and $D^\prime[m]\in \mathcal{T}^{\leq -1-m}$.
    Set $D_{m+1} := D^\prime[m]$. 
    Complete the composite $g\circ f \colon F\to D^\prime[m]$ to a distinguished triangle,
    \begin{displaymath}
        F\to D^\prime[m] \to R \to F[1].
    \end{displaymath}
    Set $E_{m+1}:=R[-1]$. 
    The cocones of $f$, $g$, and $g\circ f$ are respectively $E_m$, $E^\prime[m]$, and $E_{m+1}$. 
    By \Cref{rmk:octahedral}\eqref{rmk:octahedralof}, this yields a distinguished triangle 
    \begin{displaymath}
        E_m \to E_{m+1} \to E^\prime [m] \to E_m [1].
    \end{displaymath}
    By \Cref{lem:closures}\eqref{lem:closures1}\eqref{lem:closures2},
    \begin{displaymath}
        \begin{aligned}
            E_{m+1} 
            & \in \overline{\langle H\rangle}^{[1-m-A,A]} \star \overline{\langle H\rangle}^{[-m-A,A-m]}
            \\& \subseteq \overline{\langle  H\rangle}^{[-m-A,A]} 
            \\&= \overline{\langle H\rangle}^{[1-(m+1)-A,A]}.
        \end{aligned}
    \end{displaymath}
    This completes the proof.
\end{proof}

\begin{lemma}
    \label{lem:hom_vanishing}
    Let $\mathcal{T}$ be a triangulated category admitting small coproducts. 
    Choose $H\in \mathcal{T}^c$.
    Denote by $\mathcal{T}^{\leq 0}$ the aisle compactly generated by $H$.
    Assume that $\operatorname{Hom}(H,H[n])=0$ for all $n\geq n_0$.
    Then $\operatorname{Hom}(H,\mathcal{T}^{\leq -n_0})=0$.
\end{lemma}

\begin{proof}
    This is \cite[Remark 1.24]{Neeman:2026}; cf.\ \cite[Lemma 3.5(1)]{Burke/Neeman/Pauwels:2023}.
    On one hand, 
    \begin{displaymath}
        \begin{aligned}
            \mathcal{T}^{\leq -n_0} 
            &= \mathcal{T}^{\leq 0}[n_0] 
            \\&= \operatorname{Coprod}(\{H[n] : n\geq n_0\}).
        \end{aligned}
    \end{displaymath}
    On the other hand, the strictly full subcategory 
    \begin{displaymath}
        \{Y \in \mathcal{T} : \operatorname{Hom}(H,Y)=0\}
    \end{displaymath}
    contains $H[n]$ for $n\geq n_0$. 
    Moreover, by \Cref{lem:closures}\eqref{lem:closures4}, this subcategory above is closed under small coproducts and extensions.
    Consequently, the subcategory above contains $\mathcal{T}^{\leq -n_0}$.
    Thus, the desired claim follows.
\end{proof}

\section{Lifts to aisles}
\label{sec:lifts}

\begin{setup}
    \label{setup:recollement}
    Consider a recollement of singly compactly generated triangulated categories: 
    \begin{displaymath}
        \begin{tikzcd}[column sep=6em]
            {\mathcal{D}_Z} & {\mathcal{D}} & {\mathcal{D}_U.}
            \arrow["{i_\ast}" description, from=1-1, to=1-2]
            \arrow["{i^\ast}"', bend right=35, from=1-2, to=1-1]
            \arrow["{i^!}", bend left=35, from=1-2, to=1-1]
            \arrow["{j^\ast}" description, from=1-2, to=1-3]
            \arrow["{j_!}"', bend right=35, from=1-3, to=1-2]
            \arrow["{j_\ast}", bend left=35, from=1-3, to=1-2]
        \end{tikzcd}
    \end{displaymath}
    For every $E\in \mathcal{D}$, there are distinguished triangles
    \begin{equation}
        \label{eq:gluing_triangles}
        \begin{aligned}
            j_! j^\ast E 
            & \xrightarrow{\epsilon_E} E 
            \xrightarrow{\eta_E} i_\ast i^\ast E 
            \xrightarrow{\varphi_E} j_! j^\ast E[1],
            \\& i_\ast i^! E \to E \to j_\ast j^\ast E \to i_\ast i^! E[1].
        \end{aligned}
    \end{equation}
    Here $\epsilon$ is the counit of $j_! \dashv j^\ast$ and $\eta$ is the unit of $i^\ast \dashv i_\ast$.
\end{setup}

\begin{lemma}
    \label{lem:glued}
    Assume \Cref{setup:recollement}.
    Consider $t$-structures $(\mathcal{D}_Z^{\leq 0},\mathcal{D}_Z^{\geq 0})$ and $(\mathcal{D}_U^{\leq 0},\mathcal{D}_U^{\geq 0})$ respectively on $\mathcal{D}_Z$ and $\mathcal{D}_U$. 
    Denote by $(\mathcal{D}^{\leq 0},\mathcal{D}^{\geq 0})$ the glued $t$-structure.
    For every $n\in \mathbb{Z}$,
    \begin{displaymath}
        \mathcal{D}^{\leq n} = \{E\in \mathcal{D} :\ i^\ast E\in \mathcal{D}_Z^{\leq n}\ \text{and}\ j^\ast E\in \mathcal{D}_U^{\leq n}\}.
    \end{displaymath}
    In particular,
    \begin{displaymath}
        i^\ast \mathcal{D}^{\leq n}\subseteq \mathcal{D}_Z^{\leq n},
        \qquad 
        j^\ast \mathcal{D}^{\leq n}\subseteq \mathcal{D}_U^{\leq n},
        \qquad 
        i_\ast \mathcal{D}_Z^{\leq n}\subseteq \mathcal{D}^{\leq n},
        \qquad 
        j_! \mathcal{D}_U^{\leq n}\subseteq \mathcal{D}^{\leq n}.
    \end{displaymath}
\end{lemma}

\begin{proof}
    If $n=0$, then this is the definition of the glued aisle.
    Since $i^\ast$ and $j^\ast$ commute with shifts, 
    \begin{displaymath}
        \mathcal{D}^{\leq n} = \{E\in \mathcal{D} :\ i^\ast E\in \mathcal{D}_Z^{\leq n}\ \text{and}\ j^\ast E\in \mathcal{D}_U^{\leq n}\}.
    \end{displaymath}
    Moreover, the functors $i^\ast, j^\ast, i_\ast, j_!$ are right $t$-exact for the $n=0$ case.
    After shifting, the desired containments follow.
\end{proof}

\begin{lemma}
    \label{lem:normalize}
    Assume \Cref{setup:recollement} where $\mathcal{D}_U$ is weakly approximable with compact generator $G_U$.
    Suppose $G^\prime$ is a compact generator for $\mathcal{D}$ which satisfies $\operatorname{Hom}(G^\prime,G^\prime[n])=0$ for $n\gg 0$.
    Denote by $\mathcal{D}_U^{\leq 0}$ the aisle compactly generated by $G_U$.
    Then the following hold:
    \begin{enumerate}
        \item \label{lem:normalize1} there exists an $m\geq 0$ such that $j^\ast G^\prime \in \mathcal{D}_U^{\leq m}$
        \item \label{lem:normalize2} $G := G^\prime[m]\oplus j_! G_U$ is a compact generator of $\mathcal{D}$, and $i^\ast G$ is a compact generator of $\mathcal{D}_Z$
        \item \label{lem:normalize3} the $t$-structure on $\mathcal{D}$ compactly generated by $G$ is the gluing of the $t$-structure on $\mathcal{D}_Z$ compactly generated by $i^\ast G$ with $(\mathcal{D}_U^{\leq 0},\mathcal{D}_U^{\geq 0})$
        \item \label{lem:normalize4} $\operatorname{Hom}(G,G[n])=0$ for $n\gg 0$.
    \end{enumerate}
\end{lemma}

\begin{proof}
    \eqref{lem:normalize1} and \eqref{lem:normalize3} are essentially from \cite[Lemma 3.11]{Burke/Neeman/Pauwels:2023}.
    The observation below is \cite[Lemma 3.5(2)]{Burke/Neeman/Pauwels:2023} (see also \cite[Lemma A.1]{GuisadoVillaalgordo/Lank:2026}). 
    We recall the argument for convenience.

    We prove \eqref{lem:normalize1}.
    Recall that $j_!$ preserves compact objects.
    Hence, $j_! G_U\in \mathcal{D}^c$.
    We observe that for all $F,F^\prime\in \mathcal{D}^c$, there exists an integer $N$ such that $\operatorname{Hom}(F,F^\prime[n])=0$ for all $n\geq N$.
    First, consider the full subcategory $\mathcal{H}$ of $F\in \mathcal{D}^c$ where such an $N$ exists with $F^\prime = G^\prime$.
    This subcategory is thick (e.g.\ use the long exact sequences of $\operatorname{Hom}(-,G^\prime[n])$). 
    By hypothesis, $G^\prime \in \mathcal{H}$.
    As $G^\prime$ is a classical generator of $\mathcal{D}^c$ (see \Cref{sec:prelim_generation}), $\mathcal{H} = \mathcal{D}^c$.
    Second, fix $F\in \mathcal{D}^c$. 
    Consider the full subcategory $\mathcal{H}^\prime$ of $F^\prime\in \mathcal{D}^c$ for which such an $N$ exists.
    Again, $\mathcal{H}^\prime$ is thick, which can be shown using the long exact sequences of $\operatorname{Hom}(F,-)$.
    Also, by the first step above, $G^\prime \in \mathcal{H}^\prime$. 
    Hence, $\mathcal{H}^\prime = \mathcal{D}^c$, which proves the observation. 
    Applying the observation to $F = j_! G_U$ and $F^\prime = G^\prime$, and using adjunction, we have for $n\gg 0$,
    \begin{displaymath}
        \operatorname{Hom}(G_U[-n],j^\ast G^\prime)\cong \operatorname{Hom}(j_! G_U[-n],G^\prime) = 0.
    \end{displaymath}
    Since $\mathcal{D}_U$ is weakly approximable and $\mathcal{D}_U^{\leq 0}$ is in the preferred equivalence class, \cite[Corollary 3.10(1)]{Burke/Neeman/Pauwels:2023} implies that $j^\ast G^\prime\in \mathcal{D}_U^{\leq m}$ for some $m\geq 0$.

    Note that $j_! G_U$ is compact. 
    Hence, the hypotheses of \cite[Lemma A.1]{GuisadoVillaalgordo/Lank:2026} hold: $G^\prime$ and $G_U$ are respectively compact generators of $\mathcal{D}$ and $\mathcal{D}_U$, $\mathcal{D}_U^{\leq 0}$ is the aisle compactly generated by $G_U$, and $j^\ast G^\prime\in \mathcal{D}_U^{\leq m}$ by \eqref{lem:normalize1}.
    Consequently, \eqref{lem:normalize2} and \eqref{lem:normalize3} follow from \cite[Lemma A.1]{GuisadoVillaalgordo/Lank:2026}.

    To prove \eqref{lem:normalize4}, apply the observation with $F=F^\prime=G$.
\end{proof}

\begin{setup}
    \label{setup:standing}
    Consider \Cref{setup:recollement}.
    Assume that $\mathcal{D}_Z$ and $\mathcal{D}_U$ are weakly approximable.
    Suppose $\mathcal{D}$ has a compact generator $G^\prime$ satisfying $\operatorname{Hom}(G^\prime,G^\prime[n])=0$ for $n\gg 0$.
    Fix a compact generator $G_U$ of $\mathcal{D}_U$.
    Choose $m$ and $G = G^\prime[m]\oplus j_! G_U$ from \Cref{lem:normalize}.
    Consider the following $t$-structures:
    \begin{itemize}
        \item $\mathcal{D}_U^{\leq 0}$ is compactly generated by $G_U$;
        \item $\mathcal{D}_Z^{\leq 0}$ is compactly generated by $i^\ast G$;
        \item $\mathcal{D}^{\leq 0}$ is the glued $t$-structure (by \Cref{lem:normalize}\eqref{lem:normalize3}, it is the $t$-structure compactly generated by $G$; hence, $G\in \mathcal{D}^{\leq 0}$).
    \end{itemize}
    Note that $\mathcal{D}_U^{\leq 0}$ and $\mathcal{D}_Z^{\leq 0}$ are respectively in the preferred equivalence classes of $\mathcal{D}_U$ and $\mathcal{D}_Z$. 
    By \cite[Lemma 3.7]{Burke/Neeman/Pauwels:2023} (equivalently, \cite[Proposition 3.6]{Neeman:2026}), \Cref{lem:hom_vanishing}, and \Cref{lem:normalize}\eqref{lem:normalize4}, there exist $A_1,A_2,A_3 \gg 0$ such that the following hold:
    \begin{enumerate}
        \item \label{setup:standinga1} $(i^\ast G,\mathcal{D}_Z^{\leq 0},A_1)$ is weak approximation data for $\mathcal{D}_Z$
        \item \label{setup:standinga2} $(G_U,\mathcal{D}_U^{\leq 0},A_2)$ is weak approximation data for $\mathcal{D}_U$
        \item \label{setup:standinga3} $\operatorname{Hom}(G,\mathcal{D}^{\leq -A_3})=0$.
    \end{enumerate}
    Set $A= \max\{A_1,A_2,A_3\}$. 
    Define $K := A+1$, $B := 3A$, and $\mathcal{P} := \overline{\langle G\rangle}^{[-B,B]} \cap \mathcal{D}^{\leq K}$.
    By \Cref{lem:closures}\eqref{lem:closures1}, and the fact that aisles are closed under extensions and small coproducts, $\mathcal{P}$ is closed under small coproducts and extensions.
    Since $j_! G_U$ is a direct summand of $G$, it follows that $j_! G_U\in \overline{\langle G\rangle}^{[0,0]}$.
\end{setup}

\begin{definition}
    \label{def:lift}
    Assume \Cref{setup:standing}. 
    Let $E^\prime \in \mathcal{D}_Z$, $Y\in \mathcal{D}$, and $f^\prime \colon E^\prime \to i^\ast Y$.
    A \textbf{$\mathcal{P}$-lift} of $f^\prime$ is a pair $(q,\theta)$ consisting of a morphism $q\colon E\to Y$ with $E\in \mathcal{P}$ and an isomorphism $\theta\colon i^\ast E \to E^\prime$ such that $i^\ast q = f^\prime \circ \theta$.
\end{definition}

\begin{lemma}
    \label{lem:onecell}
    Assume \Cref{setup:standing}. 
    Let $Y\in \mathcal{D}^{\leq K}$, $-A-1\leq n\leq A$, and $f^\prime \colon (i^\ast G)[n]\to i^\ast Y$.
    Then $f^\prime$ admits a $\mathcal{P}$-lift.
\end{lemma}

\begin{proof}
    The following proof is a variant of \cite[Lemma 5.3]{Burke/Neeman/Pauwels:2023}.
    Specifically, the lift is constructed directly from the adjunct $g$ rather than from a roof \cite[Lemma 5.2]{Burke/Neeman/Pauwels:2023}, and it lies in $\mathcal{D}^{\leq K}$ for the same $K$ that bounds $Y$.

    Since $i^\ast$ is exact, $i^\ast(G[n])$ can be identified with $(i^\ast G)[n]$.
    Hence, $f^\prime$ yields a morphism $i^\ast(G[n])\to i^\ast Y$.
    Denote by $\eta$ the unit and $\overline{\varepsilon}$ the counit of $i^\ast\dashv i_\ast$.

    Set $M := j^\ast Y$.
    By \Cref{lem:glued}, $M\in \mathcal{D}_U^{\leq K}$, and so $M[K]\in \mathcal{D}_U^{\leq 0}$.
    Apply \Cref{lem:iterate} to the weak approximation data $(G_U,\mathcal{D}_U^{\leq 0},A)$ of \Cref{setup:standing}\eqref{setup:standinga2}, $M[K]$, and $3A+1 = K+2A$ in place of $m$.
    This gives a distinguished triangle
    \begin{displaymath}
        W^\prime\to M[K]\to V^\prime\to W^\prime[1].
    \end{displaymath}
    Here $W^\prime\in \overline{\langle G_U\rangle}^{[-4A,A]}\cap \mathcal{D}_U^{\leq 0}$ and $V^\prime\in \mathcal{D}_U^{\leq -3A-1}$.
    Define $W := W^\prime[-K]$ and $V := V^\prime[-K]$.
    Shifting by $[-K]$ yields a distinguished triangle
    \begin{equation}
        \label{eq:onecell_wc}
        W\xrightarrow{w} M\xrightarrow{c} V\to W[1].
    \end{equation}
    By \Cref{lem:closures}\eqref{lem:closures2},
    \begin{displaymath}
        W\in \overline{\langle G_U\rangle}^{[-4A+K,A+K]} = \overline{\langle G_U\rangle}^{[1-3A,2A+1]}.
    \end{displaymath}
    Moreover, $W\in \mathcal{D}_U^{\leq 0}[-K] = \mathcal{D}_U^{\leq K}$ and $V\in \mathcal{D}_U^{\leq -3A-1}[-K] = \mathcal{D}_U^{\leq -2A}$.

    Now let $g\colon G[n]\to i_\ast i^\ast Y$ be the adjunct of $f^\prime$.
    That is, $g = i_\ast(f^\prime)\circ \eta_{G[n]}$, and so $f^\prime = \overline{\varepsilon}_{i^\ast Y}\circ i^\ast(g)$.
    Set
    \begin{displaymath}
        \varphi := \varphi_Y\colon i_\ast i^\ast Y\to j_! M[1]
    \end{displaymath}
    to be the morphism in \eqref{eq:gluing_triangles}, and $\xi := \varphi\circ g$. 
    By \Cref{lem:glued}, $j_! V\in \mathcal{D}^{\leq -2A}$.
    As $n\leq A$, 
    \begin{displaymath}
        j_! V[1-n]\in \mathcal{D}^{\leq n-1-2A}\subseteq \mathcal{D}^{\leq -A}.
    \end{displaymath}
    Hence, \Cref{setup:standing}\eqref{setup:standinga3} gives
    \begin{equation}
        \label{eq:onecell_vanishing}
        \operatorname{Hom}(G[n],j_! V[1])\cong \operatorname{Hom}(G,j_! V[1-n]) = 0.
    \end{equation}
    Applying $j_!$ to \eqref{eq:onecell_wc} gives a distinguished triangle
    \begin{displaymath}
        j_! W\xrightarrow{j_!(w)} j_! M \xrightarrow{j_!(c)} j_! V\to j_! W[1].
    \end{displaymath}
    By \cite[Lemma 1.1.10]{Neeman:2001}, $\operatorname{Hom}(G[n-1],-)$ is a homological functor on $\mathcal{D}$.
    This yields an exact sequence
    \begin{displaymath}
        \operatorname{Hom}(G[n-1],j_! W)
        \xrightarrow{j_!(w)\circ -} \operatorname{Hom}(G[n-1],j_! M)
        \xrightarrow{j_!(c)\circ -} \operatorname{Hom}(G[n-1],j_! V).
    \end{displaymath}
    By \eqref{eq:onecell_vanishing}, 
    \begin{displaymath}
        \operatorname{Hom}(G[n-1],j_! V)\cong \operatorname{Hom}(G[n],j_! V[1]) = 0.
    \end{displaymath}
    Hence, the left-hand map is surjective.
    Thus, there exists $\xi^{\prime\prime}\colon G[n-1]\to j_! W$ such that $\xi[-1] = j_!(w)\circ \xi^{\prime\prime}$.
    Setting $\xi^\prime := \xi^{\prime\prime}[1]\colon G[n]\to j_! W[1]$, we get $\xi = j_!(w)[1]\circ \xi^\prime$.

    Next we construct the lift.
    Completing $\xi^\prime$ and rotating backwards twice yields a distinguished triangle
    \begin{equation}
        \label{eq:onecell_triangle}
        j_! W\xrightarrow{\alpha} Q\xrightarrow{\beta} G[n]\xrightarrow{\xi^\prime} j_! W[1].
    \end{equation}
    Consider the diagram
    \begin{displaymath}
        \begin{tikzcd}
            j_! W \ar[r, "\alpha"] \ar[d, "j_!(w)"'] & Q \ar[r, "\beta"] \ar[d, dashed, "q"] & G[n] \ar[r, "\xi^\prime"] \ar[d, "g"] & j_! W[1] \ar[d, "{j_!(w)[1]}"] \\
            j_! M \ar[r, "\epsilon_Y"] & Y \ar[r, "\eta_Y"] & i_\ast i^\ast Y \ar[r, "\varphi"] & j_! M [1]
        \end{tikzcd}
    \end{displaymath}
    with distinguished rows.
    Indeed, the top row is \eqref{eq:onecell_triangle}, whereas the bottom row is the first triangle of \eqref{eq:gluing_triangles} for $E = Y$ (i.e.\ use that $M = j^\ast Y$).
    The right-hand square commutes because 
    \begin{displaymath}
        \varphi\circ g = \xi = j_!(w)[1]\circ \xi^\prime.
    \end{displaymath}
    By TR3, applied after rotating both rows backwards, there exists $q\colon Q\to Y$ making the diagram commute.
    In particular, $\eta_Y\circ q = g\circ \beta$.

    Set $\theta := i^\ast\beta\colon i^\ast Q\to i^\ast(G[n]) = (i^\ast G)[n]$.
    Recall that $i^\ast$ is exact and $i^\ast j_! = 0$.
    Applying $i^\ast$ to \eqref{eq:onecell_triangle} gives a distinguished triangle
    \begin{displaymath}
        0\to i^\ast Q\xrightarrow{\theta} i^\ast(G[n])\to 0[1],
    \end{displaymath}
    and so $\theta$ is an isomorphism.
    Moreover, applying $i^\ast$ to $\eta_Y\circ q = g\circ \beta$ and using the triangle identity $\overline{\varepsilon}_{i^\ast Y}\circ i^\ast(\eta_Y) = 1_{i^\ast Y}$, it follows that
    \begin{displaymath}
        i^\ast q
        = \overline{\varepsilon}_{i^\ast Y}\circ i^\ast(\eta_Y)\circ i^\ast q
        = \overline{\varepsilon}_{i^\ast Y}\circ i^\ast(g)\circ i^\ast\beta
        = f^\prime\circ \theta.
    \end{displaymath}

    Lastly, we check that $Q$ lies in $\mathcal{P}$.
    By \eqref{eq:onecell_triangle}, $Q\in j_! W\star G[n]$.
    Since $\mathcal{P}$ is closed under extensions, it suffices to show that $j_! W, G[n]\in \mathcal{P}$.
    The functor $j_!$ is exact, preserves small coproducts, and satisfies $j_! G_U\in \overline{\langle G\rangle}^{[0,0]}$ by \Cref{setup:standing}.
    By \Cref{lem:closures}\eqref{lem:closures3}, $j_! W\in \overline{\langle G\rangle}^{[1-3A,2A+1]}$.
    As $A\geq 1$, we have $2A+1\leq 3A = B$, so \Cref{lem:closures}\eqref{lem:closures2} gives $j_! W\in \overline{\langle G\rangle}^{[-B,B]}$.
    Moreover, by \Cref{lem:glued}, $W\in \mathcal{D}_U^{\leq K}$ implies $j_! W\in \mathcal{D}^{\leq K}$.
    Since $-A-1\leq n\leq A$, one has $-A\leq -n\leq A+1$, and so $G[n] = G[-(-n)]\in G[-A,A+1]$.
    Applying \Cref{lem:closures}\eqref{lem:closures2} with $A+1\leq 3A$, $G[n]\in \overline{\langle G\rangle}^{[-A,A+1]}\subseteq \overline{\langle G\rangle}^{[-B,B]}$.
    As $G\in \mathcal{D}^{\leq 0}$ by \Cref{setup:standing}, $G[n]\in \mathcal{D}^{\leq -n}\subseteq \mathcal{D}^{\leq A+1} = \mathcal{D}^{\leq K}$.
    Thus, $Q\in \mathcal{P}$, and so $(q,\theta)$ is a $\mathcal{P}$-lift of $f^\prime$.
\end{proof}

\begin{notation}
    Assume \Cref{setup:standing}. 
    Denote by $\mathcal{C}\subseteq \mathcal{D}_Z$ the full subcategory of objects $E^\prime$ such that every morphism $f^\prime \colon E^\prime \to i^\ast Y$ with $Y\in \mathcal{D}^{\leq K}$ admits a $\mathcal{P}$-lift.
\end{notation}

\begin{proposition}
    \label{prop:closure}
    Assume \Cref{setup:standing}. 
    The subcategory $\mathcal{C}$ is strictly full, closed under small coproducts and extensions, and contains $(i^\ast G)[n]$ for $-A-1\leq n\leq A$.
    Consequently,
    \begin{displaymath}
        \overline{\langle i^\ast G\rangle}^{[-A,A]}\subseteq \operatorname{Coprod}((i^\ast G)[-A,A+1])\subseteq \mathcal{C}.
    \end{displaymath}
\end{proposition}

\begin{proof}
    First we show that $\mathcal{C}$ is strictly full.
    Let $\sigma\colon E^\prime\to E^{\prime\prime}$ be an isomorphism with $E^{\prime\prime}\in \mathcal{C}$.
    Fix $f^\prime\colon E^\prime\to i^\ast Y$ with $Y\in \mathcal{D}^{\leq K}$.
    Since $E^{\prime\prime}\in \mathcal{C}$, $f^\prime\circ \sigma^{-1}$ has a $\mathcal{P}$-lift $(q,\theta)$.
    Then $\sigma^{-1}\circ \theta$ is an isomorphism, and $i^\ast q = f^\prime\circ (\sigma^{-1}\circ \theta)$.
    Thus, $(q,\sigma^{-1}\circ \theta)$ is a $\mathcal{P}$-lift of $f^\prime$. Hence, $E^\prime\in \mathcal{C}$.

    By \Cref{lem:onecell}, $(i^\ast G)[n] \in \mathcal{C}$  for $-A-1\leq n\leq A$.

    Next we show that $\mathcal{C}$ is closed under small coproducts.
    Choose a set $\Lambda$.
    Pick some $E^\prime_\lambda\in \mathcal{C}$ for each $\lambda\in \Lambda$.
    Put $E^\prime := \bigoplus_{\lambda\in \Lambda} E^\prime_\lambda$.
    Denote by $\iota_\lambda\colon E^\prime_\lambda\to E^\prime$ the inclusions.
    Fix $f^\prime\colon E^\prime\to i^\ast Y$ such that $Y\in \mathcal{D}^{\leq K}$.
    For each $\lambda \in \Lambda$, there exists a $\mathcal{P}$-lift $(q_\lambda\colon E_\lambda\to Y,\theta_\lambda)$ of $f^\prime\circ \iota_\lambda$.
    Define $E := \bigoplus_{\lambda\in \Lambda} E_\lambda$.
    Set $\kappa_\lambda\colon E_\lambda\to E$ to be the inclusions for each $\lambda \in \Lambda$.
    Since $\mathcal{P}$ is closed under small coproducts, $E\in \mathcal{P}$.
    Denote by $q\colon E\to Y$ the morphism with $q\circ \kappa_\lambda = q_\lambda$ for each $\lambda \in \Lambda$.
    As $i^\ast$ preserves small coproducts, $i^\ast E$ is the coproduct of the $i^\ast E_\lambda$ with inclusions $i^\ast(\kappa_\lambda)$ for $\lambda \in \Lambda$.
    Hence, there exists $\theta\colon i^\ast E\to E^\prime$ with $\theta\circ i^\ast(\kappa_\lambda) = \iota_\lambda\circ \theta_\lambda$ for each $\lambda \in \Lambda$.
    Note that $\theta$ is an isomorphism because it is a coproduct of the isomorphisms $\theta_\lambda$.
    For each $\lambda \in \Lambda$,
    \begin{displaymath}
        i^\ast q\circ i^\ast(\kappa_\lambda)
        = i^\ast(q_\lambda)
        = f^\prime\circ \iota_\lambda\circ \theta_\lambda
        = f^\prime\circ \theta\circ i^\ast(\kappa_\lambda).
    \end{displaymath}
    Hence, by the universal property of coproducts, $i^\ast q = f^\prime\circ \theta$.
    Thus, $(q,\theta)$ is a $\mathcal{P}$-lift of $f^\prime$.

    Now we prove that $\mathcal{C}$ is closed under extensions.
    Consider a distinguished triangle
    \begin{displaymath}
        E^\prime_1\xrightarrow{u^\prime} E^\prime\to E^\prime_2\to E^\prime_1[1]
    \end{displaymath}
    where $E^\prime_1,E^\prime_2\in \mathcal{C}$.
    Fix $f^\prime\colon E^\prime\to i^\ast Y$ with $Y\in \mathcal{D}^{\leq K}$.
    We will lift $f^\prime$ in two stages: first on $E^\prime_1$, and then on $E^\prime_2$, into the cone of the first lift.

    We start with $E^\prime_1$.
    As $E^\prime_1\in \mathcal{C}$, $f^\prime_1 := f^\prime\circ u^\prime$ has a $\mathcal{P}$-lift $(q_1\colon E_1\to Y,\theta_1)$.
    Complete to a distinguished triangle
    \begin{displaymath}
        E_1\xrightarrow{q_1} Y\xrightarrow{r} Y^\prime\xrightarrow{s} E_1[1].
    \end{displaymath}
    Since $E_1\in \mathcal{P} \subseteq \mathcal{D}^{\leq K}$, $E_1[1]\in \mathcal{D}^{\leq K-1}\subseteq \mathcal{D}^{\leq K}$.
    As aisles are closed under extensions, $Y^\prime\in Y\star E_1[1]\subseteq \mathcal{D}^{\leq K}$.
    Applying $i^\ast$ to this distinguished triangle gives the top row of the diagram
    \begin{displaymath}
        \begin{tikzcd}[column sep=large]
            i^\ast E_1 \ar[r, "i^\ast q_1"] \ar[d, "\theta_1"'] & i^\ast Y \ar[r, "i^\ast r"] \ar[d, equal] & i^\ast Y^\prime \ar[r, "i^\ast s"] \ar[d, equal] & (i^\ast E_1)[1] \ar[d, "{\theta_1[1]}"] \\
            E^\prime_1 \ar[r, "f^\prime_1"] & i^\ast Y \ar[r, "i^\ast r"] & i^\ast Y^\prime \ar[r, "{\theta_1[1]\circ i^\ast s}"] & E^\prime_1[1].
        \end{tikzcd}
    \end{displaymath}
    Since $i^\ast q_1 = f^\prime_1\circ \theta_1$, the left square commutes.
    By construction, the other two squares commute.
    As all vertical maps are isomorphisms, the bottom row is isomorphic to a distinguished triangle, and hence is distinguished.

    Next we find the morphism that $f^\prime$ induces on $E^\prime_2$.
    Again, we complete to a distinguished triangle
    \begin{displaymath}
        E^\prime\xrightarrow{f^\prime} i^\ast Y\xrightarrow{p^\prime} R^\prime\to E^\prime[1].
    \end{displaymath}
    We apply \Cref{rmk:octahedral}\eqref{rmk:octahedraloc} to $E^\prime_1\xrightarrow{u^\prime} E^\prime\xrightarrow{f^\prime} i^\ast Y$.
    For this, we use the given distinguished triangle on $u^\prime$, the chosen distinguished triangle on $f^\prime$, and the bottom row of the diagram above as the distinguished triangle on $f^\prime\circ u^\prime = f^\prime_1$.
    This gives a distinguished triangle
    \begin{displaymath}
        E^\prime_2\xrightarrow{f^\prime_2} i^\ast Y^\prime\xrightarrow{h} R^\prime\to E^\prime_2[1]
    \end{displaymath}
    such that the following diagram commutes:
    \begin{displaymath}
        \begin{tikzcd}
            i^\ast Y \ar[r, "i^\ast r"] \ar[dr, "p^\prime"'] & i^\ast Y^\prime \ar[d, "h"] \\
            & R^\prime.
        \end{tikzcd}
    \end{displaymath}

    Now we lift $f^\prime_2$.
    Since $E^\prime_2\in \mathcal{C}$ and $Y^\prime\in \mathcal{D}^{\leq K}$, $f^\prime_2$ has a $\mathcal{P}$-lift $(q_2\colon E_2\to Y^\prime,\theta_2)$.
    Choose a distinguished triangle
    \begin{displaymath}
        E_2\xrightarrow{q_2} Y^\prime\xrightarrow{s_2} Y^{\prime\prime}\to E_2[1].
    \end{displaymath}

    Next we combine the two lifts.
    Complete to a distinguished triangle
    \begin{displaymath}
        E\xrightarrow{q} Y\xrightarrow{s_2\circ r} Y^{\prime\prime}\to E[1].
    \end{displaymath}
    The cocones of $r$, $s_2$, and $s_2\circ r$ are respectively $E_1$, $E_2$, and $E$.
    Hence, \Cref{rmk:octahedral}\eqref{rmk:octahedralof} gives a distinguished triangle $E_1\to E\to E_2\to E_1[1]$.
    Since $\mathcal{P}$ is closed under extensions, $E\in E_1\star E_2\subseteq \mathcal{P}$.

    It remains to compare $i^\ast q$ with $f^\prime$.
    We do this in two steps.
    First, applying $i^\ast$ to the distinguished triangle on $q_2$ gives the top row of the diagram
    \begin{displaymath}
        \begin{tikzcd}[column sep=large]
            i^\ast E_2 \ar[r, "i^\ast q_2"] \ar[d, "\theta_2"'] & i^\ast Y^\prime \ar[r, "i^\ast s_2"] \ar[d, equal] & i^\ast Y^{\prime\prime} \ar[r] \ar[d, dashed, "\chi"] & (i^\ast E_2)[1] \ar[d, "{\theta_2[1]}"] \\
            E^\prime_2 \ar[r, "f^\prime_2"] & i^\ast Y^\prime \ar[r, "h"] & R^\prime \ar[r] & E^\prime_2[1].
        \end{tikzcd}
    \end{displaymath}
    Note that the bottom row is the distinguished triangle on $f^\prime_2$ found above.
    As $i^\ast q_2 = f^\prime_2\circ \theta_2$, the left square commutes.
    By TR3, there exists a morphism $\chi\colon i^\ast Y^{\prime\prime}\to R^\prime$ making the diagram commute.
    Hence, $\chi\circ i^\ast s_2 = h$.
    Since $\theta_2$ and the identity are isomorphisms, so is  $\chi$.
    Combine with $h\circ i^\ast r = p^\prime$ to obtain
    \begin{displaymath}
        \chi\circ i^\ast(s_2\circ r) = h\circ i^\ast r = p^\prime.
    \end{displaymath}
    Applying $i^\ast$ to the distinguished triangle on $q$ gives 
    \begin{displaymath}
        \begin{tikzcd}[column sep=large]
            i^\ast E \ar[r, "i^\ast q"] \ar[d, dashed, "\theta"'] & i^\ast Y \ar[r, "i^\ast(s_2\circ r)"] \ar[d, equal] & i^\ast Y^{\prime\prime} \ar[r] \ar[d, "\chi"] & (i^\ast E)[1] \ar[d, dashed, "{\theta[1]}"] \\
            E^\prime \ar[r, "f^\prime"] & i^\ast Y \ar[r, "p^\prime"] & R^\prime \ar[r] & E^\prime[1].
        \end{tikzcd}
    \end{displaymath}
    The bottom row is the distinguished triangle on $f^\prime$, whereas the middle square commutes by the diagram above.
    By TR3, applied after rotating both rows, there exists a morphism $\theta\colon i^\ast E\to E^\prime$ making the diagram commute.
    Hence, $f^\prime\circ \theta = i^\ast q$.
    As before, $\theta$ is an isomorphism, because $1$ and $\chi$ are.
    Since $E\in \mathcal{P}$, $(q,\theta)$ is a $\mathcal{P}$-lift of $f^\prime$.
    Hence, $E^\prime\in \mathcal{C}$.

    Finally, we prove the two inclusions.
    The objects of $(i^\ast G)[-A,A+1]$ are $(i^\ast G)[-k]$ with $-A\leq k\leq A+1$.
    For $n = -k$, these are the objects $(i^\ast G)[n]$ with $-A-1\leq n\leq A$.
    Hence, by \Cref{lem:onecell}, these lie in $\mathcal{C}$.
    Moreover, by definition, $\operatorname{Coprod}((i^\ast G)[-A,A+1])$ is the smallest strictly full subcategory containing these objects that is closed under small coproducts and extensions.
    Since $\mathcal{C}$ has all these properties, $\operatorname{Coprod}((i^\ast G)[-A,A+1])\subseteq \mathcal{C}$.
    The first inclusion is \Cref{lem:closures}\eqref{lem:closures5}.
\end{proof}

\begin{remark}
    \Cref{prop:closure} replaces the induction on the number of cones in \cite[Lemmas 5.4 and 5.5]{Burke/Neeman/Pauwels:2023}, which requires approximability.
\end{remark}

\section{Proofs of results}
\label{sec:proof_of_results}

\begin{proof}
    [Proof of \Cref{thm:glueing_weak_approx}]
    First, let $\mathcal{D}$ be weakly approximable.
    Denote by $\mathcal{D}^{\leq 0}$ the aisle compactly generated by $G^\prime$.
    Note that $G^\prime\in \mathcal{D}^{\leq 0}$.
    By \cite[Lemma 3.7]{Burke/Neeman/Pauwels:2023}, there exists an integer $A^\prime>0$ such that $(G^\prime,\mathcal{D}^{\leq 0},A^\prime)$ satisfies \Cref{def:weak_approx}\eqref{def:weak_approx1}.
    It follows that $\operatorname{Hom}(G^\prime,\mathcal{D}^{\leq -A^\prime})=0$.
    For $n\geq A^\prime$, $G^\prime[n]\in \mathcal{D}^{\leq -n}\subseteq \mathcal{D}^{\leq -A^\prime}$.
    Hence, for all $n\geq A^\prime$, $\operatorname{Hom}(G^\prime,G^\prime[n])=0$.
    This argument applies to every compact generator of $\mathcal{D}$.

    Conversely, let $\operatorname{Hom}(G^\prime,G^\prime[n])=0$ for $n\gg 0$.
    Hence, we are in the setting of \Cref{setup:standing}, and we use its notation. 
    By \Cref{lem:normalize}\eqref{lem:normalize2}, $G$ is a compact generator of $\mathcal{D}$.
    We show that $(G,\mathcal{D}^{\leq 0},3A)$ is weak approximation data for $\mathcal{D}$.

    We start with \Cref{def:weak_approx}\eqref{def:weak_approx1}.
    Since $G\in \mathcal{D}^{\leq 0}$, $G\in \mathcal{D}^{\leq 3A}$.
    As $\mathcal{D}^{\leq -3A}\subseteq \mathcal{D}^{\leq -A}$, \Cref{setup:standing}\eqref{setup:standinga3} implies $\operatorname{Hom}(G,\mathcal{D}^{\leq -3A})=0$.

    Now we verify \Cref{def:weak_approx}\eqref{def:weak_approx2}.
    Fix $X\in \mathcal{D}^{\leq 0}$.
    We will approximate $X$ in two stages: first its image $i^\ast X$ in $\mathcal{D}_Z$, and then what remains in $\mathcal{D}_U$.

    We start on $\mathcal{D}_Z$. 
    By \Cref{lem:glued}, $i^\ast X\in \mathcal{D}_Z^{\leq 0}$.
    By \Cref{setup:standing}\eqref{setup:standinga1}, there exists a distinguished triangle
    \begin{displaymath}
        E^\prime\xrightarrow{f^\prime} i^\ast X\to D^\prime\to E^\prime[1]
    \end{displaymath}
    where $E^\prime\in \overline{\langle i^\ast G\rangle}^{[-A,A]}$ and $D^\prime\in \mathcal{D}_Z^{\leq -1}$.
    By \Cref{prop:closure}, $E^\prime\in \mathcal{C}$.
    Since $X\in \mathcal{D}^{\leq 0}\subseteq \mathcal{D}^{\leq K}$, $f^\prime$ has a $\mathcal{P}$-lift $(q_1\colon E_1\to X,\theta)$.
    Complete to a distinguished triangle
    \begin{displaymath}
        E_1\xrightarrow{q_1} X\xrightarrow{r} N\to E_1[1].
    \end{displaymath}
    Since $E_1\in \mathcal{P}\subseteq \mathcal{D}^{\leq K}$, $E_1[1]\in \mathcal{D}^{\leq K-1}\subseteq \mathcal{D}^{\leq K}$.
    As aisles are closed under extensions, $N\in X\star E_1[1]\subseteq \mathcal{D}^{\leq K}$.
    Applying $i^\ast$ to the distinguished triangle on $q_1$ gives the top row of the diagram
    \begin{displaymath}
        \begin{tikzcd}[column sep=large]
            i^\ast E_1 \ar[r, "i^\ast q_1"] \ar[d, "\theta"'] & i^\ast X \ar[r, "i^\ast r"] \ar[d, equal] & i^\ast N \ar[r] \ar[d, dashed, "\mu"] & (i^\ast E_1)[1] \ar[d, "{\theta[1]}"] \\
            E^\prime \ar[r, "f^\prime"] & i^\ast X \ar[r] & D^\prime \ar[r] & E^\prime[1],
        \end{tikzcd}
    \end{displaymath}
    whose bottom row is the distinguished triangle on $f^\prime$.
    The left square commutes because $i^\ast q_1 = f^\prime\circ \theta$.
    By TR3, there exists $\mu\colon i^\ast N\to D^\prime$ making the diagram commute.
    Since $\theta$ and the identity are isomorphisms, so is $\mu$.
    By definition, aisles are strictly full, and so $i^\ast N\in \mathcal{D}_Z^{\leq -1}$.

    Next we treat the case on $\mathcal{D}_U$.
    By \Cref{lem:glued}, $j^\ast N\in \mathcal{D}_U^{\leq K}$, and so $(j^\ast N)[K]\in \mathcal{D}_U^{\leq 0}$.
    Apply \Cref{lem:iterate} to the weak approximation data $(G_U,\mathcal{D}_U^{\leq 0},A)$ of \Cref{setup:standing}\eqref{setup:standinga2} and the object $(j^\ast N)[K]$, with $K+1$ in place of $m$.
    This gives a distinguished triangle
    \begin{displaymath}
        V^\prime\to (j^\ast N)[K]\to R^\prime\to V^\prime[1]
    \end{displaymath}
    where $V^\prime\in \overline{\langle G_U\rangle}^{[-K-A,A]}$ and $R^\prime\in \mathcal{D}_U^{\leq -K-1}$.
    Define $V := V^\prime[-K]$ and $R := R^\prime[-K]$.
    Shifting by $[-K]$ yields a distinguished triangle
    \begin{displaymath}
        V\xrightarrow{c} j^\ast N\to R\to V[1].
    \end{displaymath}
    By \Cref{lem:closures}\eqref{lem:closures2},
    \begin{displaymath}
        V\in \overline{\langle G_U\rangle}^{[-A,A+K]} = \overline{\langle G_U\rangle}^{[-A,2A+1]}.
    \end{displaymath}
    Moreover, $R\in \mathcal{D}_U^{\leq -K-1}[-K] = \mathcal{D}_U^{\leq -1}$.

    Consider the composite
    \begin{displaymath}
        \begin{tikzcd}
            j_! V \ar[r, "j_!(c)"] \ar[rr, bend left=30, "\psi"] & j_! j^\ast N \ar[r, "\epsilon_N"] & N.
        \end{tikzcd}
    \end{displaymath}
    Complete to a distinguished triangle
    \begin{displaymath}
        j_! V\xrightarrow{\psi} N \xrightarrow{\rho} D^{\prime\prime}\to j_! V[1].
    \end{displaymath}
    We claim that $D^{\prime\prime}\in \mathcal{D}^{\leq -1}$.
    Since $j_!$ is exact, applying it to the distinguished triangle on $c$ gives a distinguished triangle
    \begin{displaymath}
        j_! V\xrightarrow{j_!(c)} j_! j^\ast N\to j_! R\to j_! V[1].
    \end{displaymath}
    Hence, $\operatorname{cone}(j_!(c)) \cong j_! R$.
    By \eqref{eq:gluing_triangles}, $\operatorname{cone}(\epsilon_N)\cong i_\ast i^\ast N$.
    Applying \Cref{rmk:octahedral}\eqref{rmk:octahedraloc} to the composite above gives a distinguished triangle
    \begin{displaymath}
        j_! R\to D^{\prime\prime}\to i_\ast i^\ast N\to j_! R[1].
    \end{displaymath}
    By \Cref{lem:glued}, $R\in \mathcal{D}_U^{\leq -1}$ implies $j_! R\in \mathcal{D}^{\leq -1}$.
    Similarly, $i^\ast N\in \mathcal{D}_Z^{\leq -1}$ implies $i_\ast i^\ast N \in \mathcal{D}^{\leq -1}$.
    As aisles are closed under extensions, $D^{\prime\prime}\in j_! R\star i_\ast i^\ast N \subseteq \mathcal{D}^{\leq -1}$.

    We now assemble the two stages.
    Consider the composite
    \begin{displaymath}
        \begin{tikzcd}
            X \ar[r, "r"] \ar[rr, bend left=30, "\rho\circ r"] & N \ar[r, "\rho"] & D^{\prime\prime}.
        \end{tikzcd}
    \end{displaymath}
    Complete to a distinguished triangle
    \begin{displaymath}
        E\to X\xrightarrow{\rho\circ r} D^{\prime\prime}\to E[1].
    \end{displaymath}
    The cocones of $r$, $\rho$, and $\rho\circ r$ are respectively $E_1$, $j_! V$, and $E$.
    Hence, \Cref{rmk:octahedral}\eqref{rmk:octahedralof} gives a distinguished triangle $E_1\to E\to j_! V\to E_1[1]$.
    Note that $E\in E_1\star j_! V$.

    Lastly we bound the range of $E$.
    Since $E_1\in \mathcal{P}$ and $B = 3A$, $E_1\in \overline{\langle G\rangle}^{[-3A,3A]}$.
    Note that $j_!$ is exact and preserves small coproducts.
    By \Cref{setup:standing}, $j_! G_U\in \overline{\langle G\rangle}^{[0,0]}$.
    By \Cref{lem:closures}\eqref{lem:closures3}, $j_! V\in \overline{\langle G\rangle}^{[-A,2A+1]}$.
    As $2A+1\leq 3A$, \Cref{lem:closures}\eqref{lem:closures2} gives $j_! V\in \overline{\langle G\rangle}^{[-3A,3A]}$.
    By \Cref{lem:closures}\eqref{lem:closures1}, $\overline{\langle G\rangle}^{[-3A,3A]}$ is closed under extensions, and so $E\in \overline{\langle G\rangle}^{[-3A,3A]}$.

    Thus, $E\to X\to D^{\prime\prime}\to E[1]$ is a distinguished triangle with $E\in \overline{\langle G\rangle}^{[-3A,3A]}$ and $D^{\prime\prime}\in \mathcal{D}^{\leq -1}$.
    As $X\in \mathcal{D}^{\leq 0}$ was arbitrary, $(G,\mathcal{D}^{\leq 0},3A)$ is weak approximation data for $\mathcal{D}$.
    Therefore, $\mathcal{D}$ is weakly approximable.
\end{proof}

\begin{remark}
    The proof of \Cref{thm:glueing_weak_approx} bootstraps that of \cite[Theorem 5.1]{Burke/Neeman/Pauwels:2023}.
    The `only if' direction is \cite[Remark 3]{Burke/Neeman/Pauwels:2023}.
    For the `if' direction, the two stages compare as follows.
    \begin{itemize}
        \item (On $\mathcal{D}_Z$) In \cite{Burke/Neeman/Pauwels:2023}, the approximation $E^\prime \to i^\ast X$ has $E^\prime \in \overline{\langle i^\ast G\rangle}^{[-A,A]}_A$ (see loc.\ cit.\ for notation). 
        By \cite[Lemma 5.5]{Burke/Neeman/Pauwels:2023}, it can be lifted.
        In fact, loc.\ cit.\ proceeds by induction on the number of cones, which requires approximability.
        Here $E^\prime$ is only known to lie in $\overline{\langle i^\ast G\rangle}^{[-A,A]}$.
        Instead, it is lifted by \Cref{prop:closure}.
        \item (On $\mathcal{D}_U$) The octahedra are those in the last paragraph of the proof of \cite[Theorem 5.1]{Burke/Neeman/Pauwels:2023}.
        Here \Cref{lem:iterate} is in place of the estimate for approximable categories used in loc.\ cit.\
    \end{itemize}
\end{remark}

\begin{remark}
    Let $X$ be a quasi-compact quasi-separated algebraic space. 
    Pick a closed subset $S\subseteq |X|$ with quasi-compact complement. 
    Consider the following subcategories
    \begin{itemize}
        \item $D^{\leq 0}(X)$ consists of objects $E\in D(X)$ such that $\mathcal{H}^j (E)=0$ for $j >0$
        \item $D^{\geq 0}(X)$ consists of objects $E\in D(X)$ such that $\mathcal{H}^j (E)=0$ for $j <0$. 
    \end{itemize}
    We call $(D^{\leq 0}(X), D^{\geq 0}(X))$ the \textbf{standard $t$-structure} on $D(X)$. 
    The standard truncations of an object $E\in D_{\operatorname{qc}}(X)$ have quasi-coherent cohomology.
    Moreover, since support is defined via cohomology sheaves, the truncations of an object supported on $S$ are supported on $S$.
    See `canonical truncations' in \cite[\href{https://stacks.math.columbia.edu/tag/0118}{Tag 0118}]{StacksProject}.
    Hence, the standard $t$-structure induces a $t$-structure on $D_{\operatorname{qc},S}(X)$, which is denoted by $(D^{\leq 0}_{\operatorname{qc},S}(X), D^{\geq 0}_{\operatorname{qc},S}(X))$.
    Write $D^?_{\operatorname{qc},S}(X)$ for $(D_{\operatorname{qc},S}(X))^?$ in \Cref{sec:prelim_t-structures_t_structures} with $?\in \{-,+,b\}$. 
\end{remark}

\begin{definition}
    \label{def:wa_spaces}
    A quasi-compact quasi-separated algebraic space $X$ satisfies \textbf{W.A.}\ if $D_{\operatorname{qc},S}(X)$ is weakly approximable for every closed subset $S\subseteq |X|$ with quasi-compact complement.
\end{definition}

\begin{lemma}
    \label{lem:perf_vanishing}
    Let $X$ be a quasi-compact quasi-separated algebraic space. 
    Choose $S\subseteq |X|$ closed with quasi-compact complement.
    If $G^\prime$ is a compact generator of $D_{\operatorname{qc},S}(X)$, then $\operatorname{Hom}(G^\prime,G^\prime[n])=0$ for $n\gg 0$.
\end{lemma}

\begin{proof}
    By \cite[\href{https://stacks.math.columbia.edu/tag/0AED}{Tag 0AED}]{StacksProject}, $G^\prime$ is perfect.
    Hence, $G^\prime \in D^-_{\operatorname{qc},S}(X)$.
    By \cite[Lemma 3.40]{GuisadoVillaalgordo/Lank:2026}, it follows that $\operatorname{Hom}(G^\prime[-t],G^\prime)=0$ for $t\gg 0$.
    Thus, $\operatorname{Hom}(G^\prime,G^\prime[n])=0$ for $n\gg 0$. 
    (See also the proof of \cite[Lemma A.8]{GuisadoVillaalgordo/Lank:2026}).
\end{proof}

\begin{proposition}
    \label{prop:A8}
    Consider an elementary distinguished square $j\colon U\to X$ and $f\colon Y\to X$ of quasi-compact quasi-separated algebraic spaces (see \cite[\href{https://stacks.math.columbia.edu/tag/08GM}{Tag 08GM}]{StacksProject}). 
    If $U$ and $Y$ satisfy W.A., then so does $X$.
\end{proposition}

\begin{proof}
    Fix a closed subset $S\subseteq |X|$ with quasi-compact complement.
    Set $T := |X|\setminus |U|$.
    By \cite[Lemma A.4]{GuisadoVillaalgordo/Lank:2026}, there exists a recollement
    \begin{displaymath}
        \begin{tikzcd}[column sep=5.5em]
            {D_{\operatorname{qc},|U|\cap S}(U)} & {D_{\operatorname{qc},S}(X)} & {D_{\operatorname{qc},S\cap T}(X)}
            \arrow["{\mathbf{R}j_\ast}" description, from=1-1, to=1-2]
            \arrow["{\mathbf{L}j^\ast}"', bend right=35, from=1-2, to=1-1]
            \arrow["{j^!_S}", bend left=35, from=1-2, to=1-1]
            \arrow["{i^!}" description, from=1-2, to=1-3]
            \arrow["{i_\ast}"', bend right=35, from=1-3, to=1-2]
            \arrow["{i^\ast}", bend left=35, from=1-3, to=1-2]
        \end{tikzcd}
    \end{displaymath}
    where $i_\ast$ is the natural inclusion.
    We apply \Cref{thm:glueing_weak_approx} to this recollement.
    In the notation of \Cref{setup:recollement}, the roles are as follows:
    \begin{itemize}
        \item $\mathcal{D}_Z = D_{\operatorname{qc},|U|\cap S}(U)$, where $\mathbf{L}j^\ast$, $\mathbf{R}j_\ast$, $j^!_S$ (of the diagram above) are respectively identified with $i^\ast$, $i_\ast$, $i^!$
        \item $\mathcal{D} = D_{\operatorname{qc},S}(X)$
        \item $\mathcal{D}_U = D_{\operatorname{qc},S\cap T}(X)$, where $i_\ast$, $i^!$, $i^\ast$ are respectively identified with $j_!$, $j^\ast$, $j_\ast$ (in \Cref{setup:recollement}).
    \end{itemize}

    First, $|U|\setminus (|U|\cap S) = |U|\cap (|X|\setminus S)$.
    Since $|U|\cap (|X|\setminus S)$ is the intersection of two quasi-compact opens of the quasi-separated space $X$, it is quasi-compact.
    As $U$ satisfies W.A., $\mathcal{D}_Z$ is weakly approximable.

    Next, apply \cite[Lemma A.6]{GuisadoVillaalgordo/Lank:2026} to $T\cap S$.
    This gives a triangulated equivalence
    \begin{displaymath}
        \mathbf{L}f^\ast\colon D_{\operatorname{qc},T\cap S}(X)\to D_{\operatorname{qc},f^{-1}(T\cap S)}(Y).
    \end{displaymath}
    Note that $|Y|\setminus f^{-1}(T\cap S) = f^{-1}(|U|\cup (|X|\setminus S))$.
    Since $Y$ is quasi-compact and $X$ is quasi-separated, $f$ is quasi-compact \cite[\href{https://stacks.math.columbia.edu/tag/073B}{Tag 073B}]{StacksProject}.
    As $|U|\cup (|X|\setminus S)$ is a quasi-compact open, $f^{-1}(|U|\cup (|X|\setminus S))$ is quasi-compact.
    Moreover, as $Y$ satisfies W.A., $D_{\operatorname{qc},f^{-1}(T\cap S)}(Y)$ is weakly approximable.
    By \cite[Remark 3.4(i)]{Canonaco/Neeman/Stellari:2024}, weak approximability is invariant under triangulated equivalence.
    Thus, $\mathcal{D}_U$ is weakly approximable.

    Observe that $|X|\setminus (S\cap T) = (|X|\setminus S)\cup |U|$, which is quasi-compact.
    Moreover, the three triangulated categories in the recollement are singly compactly generated \cite[\href{https://stacks.math.columbia.edu/tag/0AED}{Tags 0AED}, \href{https://stacks.math.columbia.edu/tag/09M8}{09M8}, \& \href{https://stacks.math.columbia.edu/tag/0AEC}{0AEC}]{StacksProject}.
    If $G^\prime$ is a compact generator of $D_{\operatorname{qc},S}(X)$, then \Cref{lem:perf_vanishing} says that $\operatorname{Hom}(G^\prime,G^\prime[n])=0$ for $n\gg 0$.
    Consequently, \Cref{thm:glueing_weak_approx} asserts that $D_{\operatorname{qc},S}(X)$ is weakly approximable.
    Therefore, as $S$ was arbitrary, $X$ satisfies W.A.
\end{proof}

\begin{lemma}
    \label{lem:characterize_T-_c}
    Let $X$ be a quasi-compact quasi-separated algebraic space. 
    Fix a closed subset $S\subseteq|X|$ with quasi-compact complement.
    Set $\mathcal{T}:=D_{\operatorname{qc},S}(X)$.
    Then $\mathcal{T}^-_c$ consists of the objects of $\mathcal{T}$ that are pseudocoherent in $D(X)$ and supported on $S$.
    If $X$ is Noetherian, $\mathcal{T}^?_c = D^?_{\operatorname{coh},S}(X)$ for $?\in \{-,b\}$.
\end{lemma}

\begin{proof}
    By \cite[Lemma A.9]{GuisadoVillaalgordo/Lank:2026}, the standard $t$-structure on $\mathcal{T}$ lies in the preferred equivalence class.
    Hence, we compute $\mathcal{T}^-_c$ and $\mathcal{T}^b_c$ with respect to it.

    Pick $E\in\mathcal{T}^-_c$ and $n>0$. 
    Choose a distinguished triangle
    \begin{displaymath}
        D\to E\to F\to D[1]
    \end{displaymath}
    with $D\in\mathcal{T}^c$ and $F\in\mathcal{T}^{\leq -n}$.
    By \cite[\href{https://stacks.math.columbia.edu/tag/0AED}{Tag 0AED}]{StacksProject}, $D$ is perfect, and hence, pseudocoherent \cite[\href{https://stacks.math.columbia.edu/tag/08G8}{Tag 08G8}]{StacksProject}.
    Since $\mathcal{H}^j(F)=0$ for $j\geq 1-n$, $F$ is $(1-n)$-pseudocoherent.
    Hence, $E$ is $(1-n)$-pseudocoherent \cite[\href{https://stacks.math.columbia.edu/tag/08FV}{Tag 08FV}]{StacksProject}.
    As $n$ was arbitrary, $E$ is pseudocoherent.

    Conversely, let $E\in\mathcal{T}$ be pseudocoherent.
    Fix $n>0$.
    The open immersion $X\setminus S\to X$ is quasi-compact \cite[\href{https://stacks.math.columbia.edu/tag/073B}{Tag 073B}]{StacksProject}.
    By \cite[\href{https://stacks.math.columbia.edu/tag/08HP}{Tags 08HP} \& \href{https://stacks.math.columbia.edu/tag/08HJ}{08HJ}]{StacksProject}, there exist a perfect complex $P$ supported on $S$ and $P\to E$ inducing isomorphisms on $\mathcal{H}^i$ for $i>1-n$ and a surjection on $\mathcal{H}^{1-n}$.
    Note that the cone of $P\to E$ lies in $\mathcal{T}^{\leq -n}$.
    By \cite[\href{https://stacks.math.columbia.edu/tag/0AED}{Tag 0AED}]{StacksProject}, $P\in\mathcal{T}^c$.
    Therefore, $E\in\mathcal{T}^-_c$, which gives the first claim.

    We show the second claim. 
    By \cite[\href{https://stacks.math.columbia.edu/tag/08IK}{Tag 08IK}]{StacksProject}, the pseudocoherent objects of $D_{\operatorname{qc}}(X)$ are exactly those of $D^-_{\operatorname{coh}}(X)$.
    Hence, $\mathcal{T}^-_c = D^-_{\operatorname{coh}}(X) \cap \mathcal{T} = D^-_{\operatorname{coh},S}(X)$, and $\mathcal{T}^b_c = \mathcal{T}^b \cap\mathcal{T}^-_c =D^b_{\operatorname{coh},S}(X)$. 
    This completes the proof.
\end{proof}

\begin{proof}
    [Proof of \Cref{cor:weak_approx_spaces}]
    By \cite[Theorem 3.2(iv)]{Neeman:2024}, every quasi-compact quasi-separated scheme satisfies W.A.\ 
    Moreover, \Cref{prop:A8} asserts that W.A.\ is preserved under elementary distinguished squares.
    Consequently, the induction principle \cite[\href{https://stacks.math.columbia.edu/tag/09IT}{Tag 09IT}]{StacksProject} yields that every quasi-compact quasi-separated algebraic space satisfies W.A.
    Taking $S=|X|$ gives the second claim.

    We prove the last claim.
    Here $X$ is Noetherian. 
    Fix a closed subset $S\subseteq|X|$ with quasi-compact complement.
    Set $\mathcal{T}:=D_{\operatorname{qc},S}(X)$.
    By \Cref{lem:characterize_T-_c}, $\mathcal{T}^-_c = D^-_{\operatorname{coh},S}(X)$ and $\mathcal{T}^b_c = D^b_{\operatorname{coh},S}(X)$. 
    The standard $t$-structure restricts to $D^-_{\operatorname{coh},S}(X)$. 
    Thus, we finish by \cite[Example 10.1.2]{Canonaco/Haesemeyer/Neeman/Stellari:2024} and the first claim.
    Indeed, by loc.\ cit., a weakly approximable category is coherent if some $t$-structure in the preferred class restricts to $\mathcal{T}^-_c$.
\end{proof}

\begin{remark}
    \label{rmk:stacks_version}
    With one additional lemma on finite flat covers, the argument of \Cref{cor:weak_approx_spaces} adapts to tame algebraic stacks.
    We do not pursue this.
\end{remark}

\begin{remark}
    \label{rmk:weak_approx_spaces}
    In 2024, Bregje Pauwels shared with the author unpublished notes containing a proof of the `if' direction of \Cref{thm:glueing_weak_approx}.
    That proof is also based on a closure argument.
    In those notes, the lift arises from \cite[Lemma 5.3]{Burke/Neeman/Pauwels:2023}. 
    The output is bounded by a larger constant than its target, and a separate step restores the bound at each extension.
    In this work, bounds are built into the lift from \Cref{lem:onecell}.
    This ensures no correction is necessary, and the constants are explicit. 
\end{remark}

\bibliographystyle{alpha}
\bibliography{mainbib}

\end{document}